\documentclass[reqno]{amsart}
\usepackage[colorlinks=true,linktoc=page]{hyperref}
\usepackage{amsthm,amsmath,amssymb}
\usepackage[percent]{overpic}
\usepackage{color}
\usepackage{subcaption}
\usepackage{cite}
\usepackage{upgreek}
\usepackage{fullpage}
\usepackage{mathrsfs}
\usepackage{todonotes}

\makeindex

\usepackage{tgbonum}

\allowdisplaybreaks

\title[A supersonic shock interrupting the classical development for QLW]{A quasilinear wave with a supersonic shock in a weak solution interrupting the classical development} 
\author[LA]{Leonardo Abbrescia$^{a}$}
\author[PB]{Pieter Blue$^{b}$}
\author[JS]{Jan Sbierski$^{c}$}
\author[JS]{Jared Speck$^{d}$}

\thanks{$^{a}$School of Mathematics, Georgia Institute of Technology, GA, USA.
\texttt{abbrescia@math.gatech.edu}}

\thanks{$^{b}$ School of Mathematics and Maxwell Institute for Mathematical Sciences, The University of Edinburgh, Edinburgh, UK. \texttt{P.Blue@ed.ac.uk}}

\thanks{$^{c}$ School of Mathematics and Maxwell Institute for Mathematical Sciences, The University of Edinburgh, Edinburgh, UK. \texttt{Jan.Sbierski@ed.ac.uk}}

\thanks{$^{c}$ Department of Mathematics, Vanderbilt University, TN, USA. \texttt{jared.speck@vanderbilt.edu}}

\newcommand{\Lunit}{L}
\newcommand{\uLunit}{\underline{L}}
\newcommand{\gfour}{\mathbf{g}}
\newcommand{\gfourextend}{\mathbf{g}}

\newtheorem{theorem}{Theorem}
\newtheorem{lemma}[theorem]{Lemma}
\newtheorem{corollary}[theorem]{Corollary}

\theoremstyle{definition}
\newtheorem{definition}[theorem]{Definition}
\newtheorem{remark}[theorem]{Remark}

\numberwithin{theorem}{section}

\newcounter{mnotecount}[section]

\newcommand{\Reals}{\mathbb{R}}

\newcommand{\vecL}{L}
\newcommand{\vecLbar}{\underline{L}}
\newcommand{\p}{\partial}
\newcommand{\R}{\mathbb{R}}
\newcommand{\charcurveL}{\gamma}

\newcommand{\di}{\mathrm{d}}

\usepackage{tikz}

\numberwithin{theorem}{section}

\newcommand{\solu}{\Phi}
\newcommand{\soluWeak}{\solu_W}
\newcommand{\soluClassical}{\solu_C}
\newcommand{\soluDomainWeak}{\soluDomain_W}
\newcommand{\soluSDP}{\solu_{\textnormal{SDP}}}
\newcommand{\soluBurgersSDP}{\soluBurgers_{\textnormal{SDP}}}
\newcommand{\soluDomainClassical}{\soluDomain_C}
\newcommand{\soluDomainAgree}{\soluDomain_A}
\newcommand{\soluDomainSDP}{\soluDomain_{\textnormal{SDP}}}
\newcommand{\CauchyHorizon}{\underline{\mathcal{C}}}
\newcommand{\singularBoundary}{\mathcal{B}}
\newcommand{\initialSingularity}{\mathcal{S}}
\newcommand{\shockwave}{\mathcal{K}}

\newcommand{\boundary}{\partial}

\newcommand{\hsfcData}{\Sigma}

\newcommand{\soluDomain}{\Omega}

\newcommand{\soluBurgers}{\Psi}
\newcommand{\dataBurgers}{\Psi_0}

\newcommand{\tSingularBoundaryFromData}{\tilde{t}_\singularBoundary}
\newcommand{\xSingularBoundaryFromData}{\tilde{x}_\singularBoundary}
\newcommand{\xSingularBoundary}{x_\singularBoundary}
\newcommand{\xRegularBoundary}{x_{\CauchyHorizon}}
\newcommand{\BurgersDomainRight}{\Omega_1}
\newcommand{\BurgersDomainLeft}{\Omega_2}
\newcommand{\BurgersDomainClassicalOnly}{\Omega_3}
\newcommand{\soluBurgersWeak}{\soluBurgers_W}
\newcommand{\soluBurgersClassical}{\soluBurgers_C}

\usepackage[normalem]{ulem}

\begin{document}
\maketitle

\begin{abstract}

 We study the Cauchy problem for classical and weak shock-forming solutions to a model quasilinear wave equation in $1+1$ dimensions arising from a convenient choice of $C^{\infty}$ initial data, which allows us to solve the equation using elementary arguments. The simplicity of our model allows us to succinctly illustrate various phenomena of geometric and analytic significance tied to shocks, which we view as a prototype for phenomena that can occur in  more general quasilinear hyperbolic PDE solutions. Previously, these phenomena 
 had only been demonstrated in far more technical work tied to the relativistic Euler equations and the compressible Euler equations e.g. \cite{Christodoulou:shockFormation,Christodoulou:shockDevelopment,AbbresciaSpeck,shkoller2024geometry}. 
 
 Our Cauchy problem admits a classical solution that blows up in finite time. The classical solution is defined in a largest possible globally hyperbolic region called a maximal globally hyperbolic development (MGHD), and its properties are tied to the intrinsic Lorentzian geometry of the equation and solution. The boundary of the MGHD contains an initial singularity,  a singular boundary along which the solution's second derivatives blow up (the solution and its first derivatives remain bounded), and a Cauchy horizon. Our main results provide the first example of a provably unique MGHD for a shock-forming quasilinear wave equation solution; it is provably unique because its boundary has a favorable global structure that we precisely describe.
 We also prove that for \underline{the same $C^{\infty}$ initial data}, the Cauchy problem admits a second kind of solution: a unique global weak entropy solution that has a shock curve separating two smooth regions. Of particular interest is our proof that the classical and weak solutions agree before the shock but \underline{differ} in a region to the future of the first singularity where both solutions are defined. 
 Our model shares features with more realistic models from fluid mechanics that have been previous locally studied,
such as exhibiting the same rate of blowup at the initial singularity, exhibiting similar behavior near the singular boundary, and exhibiting the same discontinuity in the higher derivatives of the weak solution across the Cauchy horizon.

Finally, we place our work in context, highlight the advantages of the geometric approach to studying solutions, and present some important open problems.

\end{abstract}

\tableofcontents

\section{Introduction} \label{S:INTRO}
In this paper, we study the Cauchy problem for the quasilinear wave (QLW) equation 

\begin{subequations} \label{E:QNLW:CAUCHY:INTRO}
    \begin{align} 
        -\partial_t^2\solu
-(\partial_t\solu-2\partial_x\solu)\partial_t\partial_x\solu
+2(2 + (\partial_t\solu-2\partial_x\solu))\partial_x^2\solu
= 0, \label{E:QNLW:INTRO} \\
        \big(\solu(0,x), \p_t \solu(0,x)\big) = (0,-\arctan(x)),\label{E:QNLW:DATA:INTRO}
    \end{align}
\end{subequations}
posed on $\Reals_t\times \Reals_x$. Our model, which can be expressed as a system of hyperbolic PDEs with a genuinely nonlinear characteristic speed, was inspired heavily by the structure of the compressible Euler equations in an appropriate formulation.\footnote{The irrotational (zero vorticity) and isentropic (constant entropy) compressible Euler equations can be expressed as a quasilinear wave equation $h^{\alpha\beta}(\partial \phi) \partial_{\alpha\beta}^2\phi = 0$ for a fluid potential $\phi$ whose gradient is at the same level of regularity as the \emph{undifferentiated} solution variables such as the fluid velocity or the density. The fluid potential $\phi$ is the analogue of our  solution $\solu$ to \eqref{E:QNLW:CAUCHY:INTRO}.} We chose the initial data for \eqref{E:QNLW:DATA:INTRO} so that the genuinely nonlinear characteristics collapse with infinite density and cause $\solu$'s second derivatives to blow up while $\solu$ and its gradient remain bounded.\footnote{In the context of compressible Euler, some authors use the terms ``gradient catastrophes'' to describe the points where the velocity and density develop infinite gradients. Since $\solu$'s second derivatives blowup while $\solu$ and its gradient remain bounded, we call this the ``initial singularity''. These are distinct from the actual shocks describing discontinuous jumps in derivatives of $\solu$.} Roughly, the first main result of the paper is:
    \begin{quote}
        We construct the largest possible \emph{classical} solution to \eqref{E:QNLW:INTRO} that is \emph{uniquely} determined by the initial data \eqref{E:QNLW:DATA:INTRO}  up until the points where the second derivatives blow up.
    \end{quote}
See Theorem\,\ref{thm:main:MGHD} for a precise statement.
The notion of ``largest possible classical solution'' is fundamentally connected to the intrinsic Lorentzian geometry of the equation and to a corresponding object known as the Maximal Globally Hyperbolic Development (MGHD for short) of the data. We will introduce these notions in detail later on.
 Roughly, the second main result of the paper is: 
    \begin{quote}
        We construct the global \emph{weak} solution to \eqref{E:QNLW:INTRO}--\eqref{E:QNLW:DATA:INTRO} that is \emph{unique} within the class of entropy solutions.
    \end{quote}
See Theorem\,\ref{T:UNIQUEWEAKSOLUTIONS} for the precise statement. Importantly, Fig. \ref{F:BOTHSOLUTIONSINONEPICTURE} shows that  \emph{there are portions of spacetime where the classical and weak solutions described above disagree}. 

\begin{figure}[h]
    \begin{overpic}[scale=1.3, grid = false, tics=5, trim=-.5cm 0cm -1cm -.5cm, clip]{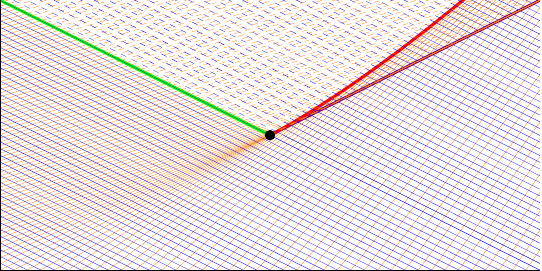}
        \put (3,42) {$t$}
        \put (92,-1) {$x$}
    \end{overpic}
    \caption{A figure of the bicharacteristics -- denoted by the orange and blue lines -- of equation \eqref{E:QNLW:INTRO}. The weak and classical solutions agree in the large region where the bicharacteristics are solid. The weak and classical solution \emph{disagree} in the wedged region where there is a family of both solid blue and dashed orange characteristics. This is due to the formation of a shock -- the purple line -- across which there are jumps only for the weak solution. The classical solution is inextendible (in a globally hyperbolic manner) to the region with only dashed bicharacteristics and only the weak solution is defined there. The rest of this paper is devoted to explaining this figure in full detail.}
    \label{F:BOTHSOLUTIONSINONEPICTURE}
\end{figure}

The uniqueness component of the first result described above is astonishingly subtle and is deeply coupled to the \underline{\emph{global}} behavior of the characteristics  of \eqref{E:QNLW:INTRO}--\eqref{E:QNLW:DATA:INTRO}. On the other hand, when reformulated in the language of \emph{Lorentzian geometric analysis}, the fundamental ideas are surprisingly easy to explain, and in $1D$, prove. This paper has two primary goals. One is to illustrate this geometric framework through the lens of the concrete model \eqref{E:QNLW:INTRO}--\eqref{E:QNLW:DATA:INTRO} to the vast community working in hyperbolic conservation laws and fluid mechanics. In particular, we hope to explain some of the foundational ideas from Christodoulou's seminal work on shock formation for the $3D$ relativistic Euler equations \cite{Christodoulou:shockFormation} -- which may be extremely difficult to parse for readers unfamiliar with general relativity --  in this current simple context. The other primary goal is to illustrate one type of shock behaviour to the general relativity and geometric wave equation communities.

We stress that for hyperbolic PDEs there are other forms of singularities that can develop from smooth initial conditions such as implosions. Although we do not address implosion singularities for our model problem, we highlight open problems in this direction Appendix  \ref{S:PROBLEMSANDPERSPECTIVE}.

After setting up some of the necessary geometric formalism in Sect.\,\ref{SS:INTRO:PRELIMINARYGEO}, we discuss the first main result concerning classical solutions in Sect.\,\ref{SS:MGHD}. We conclude the introduction with a discussion of the second main result concerning weak solutions in Sect.\,\ref{SS:GLOBALWEAK}.

\subsection{Preliminary geometric set up} \label{SS:INTRO:PRELIMINARYGEO}

The starting point of the discussion is to rewrite \eqref{E:QNLW:INTRO} in the form
\begin{align}
    (\gfour^{-1})^{\mu \nu} \partial_\mu \partial_\nu \Phi =0, \label{E:INTRO:GEOMETRICQNLW}
\end{align}
where, by setting 
    \begin{align}
        \Psi := \partial_t \Phi - 2 \partial_x \Phi, \label{E:DEFINITIONOFPSI}
    \end{align} the principal symbol $\gfour^{-1} = \begin{pmatrix}
    -1 & -\frac{1}{2}\Psi \\ -\frac{1}{2}\Psi & 2 (2 + \Psi)
\end{pmatrix}$ can be easily read off. For $\Psi \neq -4$ this matrix can be inverted to give
\begin{equation} \label{EqAcMetric}
     \gfour = \gfour(\soluBurgers) =  - \frac{8 (2+\soluBurgers)}{( 4+ \soluBurgers)^2} \, \mathrm{d}t^2 - \frac{2\soluBurgers}{(4 + \soluBurgers)^2} \, (\mathrm{d}t \otimes \mathrm{d}x + \mathrm{d}x \otimes \mathrm{d}t) + \frac{4}{(4+ \soluBurgers)^2} \, \mathrm{d}x^2 \;,
\end{equation}
which is a Lorentzian metric. In analogy with the reduction of the irrotational relativistic Euler equations to a quasilinear wave equation of the form \eqref{E:INTRO:GEOMETRICQNLW}, known since the 1930s \cite{SyngeRelHydro}, we will refer to $\gfour$ as the \textbf{acoustical metric}. The link between the method of characteristics of our  PDE\footnote{For a standard exposition on the method of characteristics for hyperbolic PDEs, see \cite{hormander2007analysisI}} and the acoustical metric $\gfour$ goes by the name of a $\gfour$-\emph{null frame}. 

\begin{definition}[The $\gfour$-null frame] \label{D:DOUBLENULLFRAME} We define $\Lunit$ and $\uLunit$ to be the vectorfields given by

\begin{align}
	\Lunit := \p_t + (2+\Psi)\p_x, & & \uLunit := \p_t - 2 \p_x. \label{E:DOUBLENULLFRAME}
\end{align}
 We call $\Lunit$ the \emph{outgoing characteristic vectorfield} and $\uLunit$ the \emph{ingoing characteristic vectorfield}. We refer to $\{\Lunit,\uLunit\}$ as the $\gfour$-null frame. We often refer to the integral curves of $\Lunit$ as the \emph{outgoing characteristics} and the integral curves of $\uLunit$ as the \emph{ingoing characteristics}.\footnote{The reader can make the analogy with $L^{(\textnormal{flat})}:= \p_t + \p_x$ and $\uLunit^{(\textnormal{flat})} := \p_t - \p_x$ for the linear wave equation $-\p_t^2 \Phi + \p_x^2 \Phi = 0$.}
 \end{definition}
 
 Direct calculations show that the Cartesian coordinate derivative vectorfields, as well as the inverse acoustical metric, may all be expressed relative to the characteristic vectorfields as

\begin{subequations}
	\begin{align}
		\p_x & = \frac{1}{4 + \Psi} (\Lunit -\uLunit), \label{E:PARTIALXINTERMSOFDOUBLENULLFRAME} \\
		\p_t & = \frac{2}{4+\Psi} \Lunit + \frac{2+\Psi}{4 + \Psi} \uLunit, \label{E:PARTIALTINTERMSOFDOUBLENULLFRAME} 
	\end{align}
\end{subequations}
	\begin{align} \label{E:INVERSEMETRICINTERMSOFDOUBLENULLFRAME}
		\gfour^{-1} = - \frac{1}{2} \Lunit \otimes \uLunit - \frac{1}{2} \uLunit \otimes \Lunit.
	\end{align}
The use of characteristic vectorfields to prove shock formation in hyperbolic PDE systems dates back to Fritz John's foundational work \cite{john1974formation}, though he did not view them through the Lorentzian point of view as in \eqref{E:INVERSEMETRICINTERMSOFDOUBLENULLFRAME}. We close this short section by noting that \eqref{E:INVERSEMETRICINTERMSOFDOUBLENULLFRAME} implies $\gfour(\Lunit,\Lunit) = \gfour(\uLunit, \uLunit) = 0$; namely, they are $\gfour$-null; see Appendix \ref{SecApp}. It is easy to see that the outgoing characteristics are the genuinely nonlinear ones.

\subsection{Maximal globally hyperbolic developments} \label{SS:MGHD}
\hfill

We now use the acoustic geometry to explain in more detail the first main result described above. The Lorentzian metric $\gfour = \gfour(\Psi) = \gfour(\p \Phi)$ given by \eqref{EqAcMetric} defines an acoustical \emph{causal structure} for solutions that
determines the notion of \textbf{globally hyperbolic development (GHD)}. A GHD $(\Omega, \Phi)$ consists of an open neighborhood $\Omega$ of $\{t=0\}\times \Reals_x$ and a smooth classical solution $\Phi: \Omega \to \Reals$ to \eqref{E:QNLW:INTRO}--\eqref{E:QNLW:DATA:INTRO} for which $(\Omega, \gfour)$ is a Lorentzian manifold with $\Omega \cap \{t=0\}$ as a Cauchy hypersurface. Roughly, this means that any point $p \in \Omega$ is in the domain of dependence of the data; 
the reader not familiar with these geometric notions is again referred to Appendix \ref{SecApp}. As a GHD describes a \emph{classical} solution of \eqref{E:QNLW:INTRO}--\eqref{E:QNLW:DATA:INTRO}, the aforementioned blow up of the second derivatives of $\solu$ occurs precisely on the \emph{boundary} of the GHD.  

A more precise statement of our first main result is that we construct the \textbf{\emph{\underline{unique} maximal\footnote{A GHD is \emph{maximal} if it is inextendible as a GHD.} globally hyperbolic development}} (MGHD) of the data up to its boundary. This is the first example of an MGHD in any spatial dimension for any system of PDEs which terminates due to the formation of a shock on its boundary. The boundary consists of the following connected components, see Fig.\,\ref{fig:MGHD_both_characteristics} for an illustration, and Theorem\,\ref{thm:main:MGHD} below for the precise result.

\begin{enumerate}
    \item[I)] An initial singularity $\initialSingularity$, which is the first point on which the second derivatives of $\solu$ blow up.
    \item[II)]
    A curve $\singularBoundary$, called the singular boundary\footnote{Christodoulou \cite{Christodoulou:shockFormation} refers to $\initialSingularity\cup\singularBoundary$ as the \emph{singular part of the boundary}.}, along which the outgoing characteristics collapse with infinite density.
    \item[III)] A curve $\CauchyHorizon$ to which the solution extends smoothly and which is null with respect to the acoustical metric. For this reason the curve is called a Cauchy horizon. 
\end{enumerate}

\begin{figure}[h]
    \begin{overpic}[scale=1, grid = false, tics=5, trim=-.5cm 0cm -1cm -.5cm, clip]{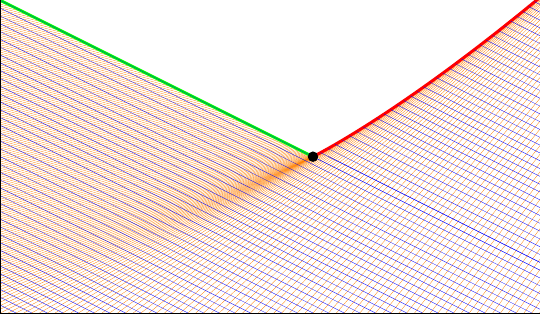}
    \end{overpic}
    \caption{An illustration of the domain and causal structure for the solutions constructed in the proof of Theorem \ref{thm:main:MGHD}. The integral curves of $\Lunit$ and $\uLunit$ are depicted in solid orange and blue, respectively. The singular boundary $\singularBoundary$ is illustrated in red and the Cauchy horizon $\CauchyHorizon$ is illustrated in green. Both emanate from the initial  singularity $\initialSingularity$, illustrated in black. The classical solution is defined to the right of the singular boundary $\singularBoundary$ and beneath the Cauchy horizon $\CauchyHorizon$.}
    \label{fig:MGHD_both_characteristics}
\end{figure}

\subsubsection{Historical context and the subtleties of unique MGHD's} \label{SSS:HISTORICALCONTEXTANDSUBTLETIESOFMGHDS}
The notion of MGHD's was first introduced in the context of General Relativity in Choquet-Bruhat--Geroch's historic work \cite{ChoquetBruhatGeroch}, which (roughly) states the following: as long as the solution of the initial value problem develops in a globally hyperbolic manner, the solution develops in a unique way. In particular, there is a unique  maximal GHD. Surprisingly, for quasilinear wave equations posed on a fixed background, global hyperbolicity of the development no longer guarantees that the solution develops in a unique manner -- and thus, in general,  the \emph{uniqueness} of MGHD's is also no longer guaranteed. The recent work \cite{EperonReallSbierski} involving the third author  provided an example of a quasilinear wave equation $h^{\alpha\beta}(\partial\Phi)\partial_{\alpha\beta}^2\Phi = 0$ and \emph{two} MGHD's $(\Omega_1,\Phi_1)$ and $(\Omega_2,\Phi_2)$ \underline{of the same initial data}, for which there are $p \in \Omega_1 \cap \Omega_2$ such that $\Phi_1(p) \neq \Phi_2(p)$. On the other hand, \cite{EperonReallSbierski} provided a \emph{global} geometric condition where, if satisfied by an MGHD, then that MGHD would be \emph{the} unique MGHD of the data. This result is teleological in the following sense: one must construct the MGHD \emph{in its entirety} to guarantee the geometric condition, which requires that the MGHD must only lie on one side of its boundary, holds.

A \emph{local} version of the global behavior described in {I)}--{II)} was first described in Christodoulou's breakthrough monograph on the formation of shocks for the relativistic Euler equations \cite{Christodoulou:shockFormation}. For small irrotational (zero vorticity) isentropic (constant entropy) symmetry breaking perturbations of a constant state, in $3D$, he gave a sharp continuation criterion which states that either a shock forms in finite time or the solution is smooth. He then constructed a large class of shock forming data and followed the solution until the first time of blow up.\footnote{The first to follow the classical evolution of a QLW until a shock forms in multiple dimensions without symmetry was Alinhac \cite{sA1999a,sA1999b}. In those works, Alinhac used Nash--Moser to overcome a loss of derivatives. However, this prevented him from analyzing the solution \emph{past} the time of first blowup to obtain the rest of the MGHD.} For a more restrictive set of shock-forming data satisfying a strict non-degeneracy assumption, he derived a local and implicit portion of the MGHD past the initial singularity up to the singular boundary $\singularBoundary$, now being co-dimension 1 hypersurfaces of Minkowski space $\Reals^{1+3}$. We refer the reader to Chapter 15 of \cite{Christodoulou:shockFormation}. Since only a small portion of the boundary of the MGHD were constructed, the uniqueness of the classical time evolution cannot be inferred from the teleological uniqueness result of MGHDs given in \cite{EperonReallSbierski}.\footnote{The authors of \cite{EperonReallSbierski} were under the mistaken impression that the full MGHD and its boundary had been constructed in \cite{Christodoulou:shockFormation}, cf.\ the comment after \cite[Lemma 4.86]{EperonReallSbierski}. } 

Since the release of \cite{Christodoulou:shockFormation}, Christodoulou's research program jump-started a remarkable surge of progress in the multi-dimensional analysis of shock formation \cite{speck2016shock,speck2016stable,miao2017formation,speck2018shock,luk2018shock,speck2019multidimensional, luk2024stability,AbbresciaSpeck,AbbresciaSpeck2,shkoller2024geometry}. We highlight the recent results of the first and fourth authors \cite{AbbresciaSpeck}. This work gave an explicit construction of a connected portion of $\singularBoundary$ for the $3D$ non-relativistic compressible Euler equations with dynamic vorticity and entropy. The class of shock-forming data \cite{AbbresciaSpeck} removes the ``strict'' non-degeneracy assumption assumed in Chapter 15 of \cite{Christodoulou:shockFormation} and allows for a \emph{full neighborhood} of plane symmetric solutions forming a non-degenerate shock. We also highlight the recent work of Shkoller-Vicol \cite{shkoller2024geometry}, which constructs \emph{both} the singular boundary and the Cauchy horizon for the $2D$ non-relativistic compressible Euler equations with dynamic vorticity but constant entropy. For an open set of strictly non-degenerate shock-forming data (analogous to the data used in Chapter 15 of \cite{Christodoulou:shockFormation}), \cite{shkoller2024geometry} considers an open set of solutions whose velocity and density were of size $O(\frac{1}{\epsilon})$. The portion of $\singularBoundary\cup\CauchyHorizon$ obtained in \cite{shkoller2024geometry} is constrained between two constant-time hypersurfaces which are of Euclidean distance $O(\epsilon)$ apart. The takeaway is that both \cite{AbbresciaSpeck,shkoller2024geometry}, which provide explicit constructive proofs of the implicit original arguments of 
\cite{Christodoulou:shockFormation}, only construct a \emph{compact} spacetime portion of the MGHD's boundary in a vicinity of the initial singularity. Fortunately, the portion of the MGHD and its boundary constructed in \cite{AbbresciaSpeck,shkoller2024geometry} \emph{does} satisfy the geometric condition of \cite{EperonReallSbierski}. However, since the speed of sound in non-relativistic Euler may be extremely large depending on the equation of state, \textbf{\emph{it does not preclude a different singularity from forming far away from the initial singularity that could potentially invalidate the geometric condition from holding.}}\footnote{On the other hand, the speed of sound for the relativistic Euler equations is always uniformly bounded by the speed of light of the ambient Minkwski space. In this case, it is possible that some sort of local uniqueness of MGHD's holds, see the list of open problems in Appendix\,\ref{S:PROBLEMSANDPERSPECTIVE}.} In other words, to prove classical determinism holds for the Cauchy problem of shock-forming data for $3D$ non-relativistic compressible Euler without symmetry, one must solve the outstanding open problem of constructing the full MGHD of the data. Similar results from the present work are extremely likely to hold for $1D$ isentropic compressible Euler flow, but analogous results in $3D$ without symmetry, especially in the presence of vorticity and entropy, will require numerous new insights due to the immense difficulties in controlling the geometry of the infinitely compressing characteristics globally across spacetime.

\subsubsection{The first main result} \label{SSS:THEFIRSTMAINRESULT}
The precise statement of our main result on the MGHD is given as follows.

\begin{theorem}[Characterization of the maximal global hyperbolic development]
\label{thm:main:MGHD}
For the initial value problem \eqref{E:QNLW:INTRO}-\eqref{E:QNLW:DATA:INTRO}, 
\begin{enumerate}
\item{} [Existence] There is a classical solution $\soluClassical$ on an open set $\soluDomainClassical\subsetneq\Reals^{1+1}$ with $\Psi_C \in(- \frac{\pi}{2}, \frac{\pi}{2})$. In particular, $\gfour$ is always a well-defined Lorentzian metric on $\soluDomainClassical$ (see \eqref{EqAcMetric}).
\label{pt:classicalProperies:existence}
\item{} [Uniqueness] $\soluDomainClassical$ is the unique maximal globally hyperbolic development of the data posed on $\{t=0\}\times\Reals$.
\label{pt:classicalProperties:MGHDuniqueness}
\item{} [Characterisation of the boundary] \label{pt:classicalProperties:characterisationOfTheBoundary} The boundary of the MGHD consists of:
    \begin{enumerate}
        \item The disjoint union of a codimension $1$ submanifold called the singular boundary $\singularBoundary$, a codimension $2$ submanifold called the initial singularity $\initialSingularity$, and a codimension $1$ submanifold $\CauchyHorizon$ called the \emph{Cauchy horizon}. Each of these are connected, and the boundary is connected. 
        \label{pt:classicalProperties:singularAndRegularBoundary}
        \item Let $\mathscr{S}(z):= (t = 1 + z^2, x = (2- \arctan(z))(1+z^2) + z)$. Then the singular boundary may be parametrized as 
        \begin{align}
            \singularBoundary =  \{ \mathscr{S}(z) \ | \ z \in (0,\infty)\}. \label{E:PARAMETRIZEDSINGULARBOUNDARY}
        \end{align}
        We also have that $\initialSingularity = \mathscr{S}(0)$.
        \label{pt:classicalProperties:parametrizeSingularBoundary}
        \item Both $\CauchyHorizon$ and $\singularBoundary$ emanate from $\initialSingularity$ in the sense that $\textnormal{cl}(\CauchyHorizon)\backslash\CauchyHorizon=\initialSingularity$ and $\textnormal{cl}(\singularBoundary)\backslash\singularBoundary=\initialSingularity$, where $\textnormal{cl(A)}$ denotes the topological closure of a set $A$. %
        \label{pt:classicalProperties:emanateFromInitialSingularity}
        \item The solution $\soluClassical$ extends smoothly to $\CauchyHorizon$. 
        \label{pt:classicalProperties:smoothAtCauchyHorizon}
        \item The solution $\soluClassical$ extends as a $C^{1,1/3}$ function on $\initialSingularity$ and as a $C^{1,1/2}$ function on $\singularBoundary$, and this regularity is sharp. Consequently, for any $p \in \initialSingularity\cup\singularBoundary$ and any sequence $\{p_n\}_{n=1}^\infty \subset \soluDomainClassical$ such that $p_n \to p$, then $|\p^2 \soluClassical(p_n)| \to \infty$.
        \label{pt:classicalProperties:regularityAtSingularBoundary}
  
        \item The acoustical metric $\gfour$ extends as a $C^{0,1/3}$, a $C^{0,1/2}$, and as a smooth Lorentzian metric to $\initialSingularity$, $\singularBoundary$, and $\CauchyHorizon$, respectively. By abuse of notation, we also denote the extended acoustical metric by $\gfour$. Similarly, $\Lunit$ extends as a $C^{0,1/3}$ and $C^{0,1/2}$ acoustically null vectorfield to $\initialSingularity$ and $ \singularBoundary$, respectively.
        \label{pt:classicalProperties:regularityOfMetric}
        \item The smooth vectorfield $\uLunit$ is everywhere tangent to $\CauchyHorizon$, making the Cauchy horizon acoustically null with respect to the extended metric. Similarly, the extended vectorfield $\Lunit$ is everywhere tangent to $\singularBoundary$, making the singular boundary intrinsically  acoustically null with respect to the extended metric $\gfourextend$.
        \label{pt:classicalProperties:boundaryIsNull}
\end{enumerate}
\item{} [Degeneracies along the singular boundary] 
 Let $p \in \singularBoundary$ and consider the backwards ODE initial value problem for the integral curves of the extended $\Lunit$:
            \begin{subequations} \label{E:ODEIVPFOREXTENDEDL}
                \begin{align}
                    \frac{d}{d s} \gamma(s) & = \Lunit \circ \gamma(s), \\
                    \gamma(0) & = p.
                \end{align}
            \end{subequations}
            Then \eqref{E:ODEIVPFOREXTENDEDL} has two different solutions $\gamma_1,\gamma_2: (-a,0] \to \Reals^{1+1}$, where $\gamma_1(-a,0] \subset \singularBoundary$ denotes the integral curve along the singular boundary emanating backwards from $p$, and $\gamma_2(-a,0] \subset \soluDomainClassical$ denotes the integral curve in the MGHD whose smooth evolution terminates at $p \in \singularBoundary$ (and is depicted by an orange characteristic in Fig.\,\ref{fig:MGHD_both_characteristics}).
            \label{PT:NONUNIQUEINTEGRALCURVES}
\end{enumerate}
\end{theorem}

 Point \eqref{PT:NONUNIQUEINTEGRALCURVES} implies that the singular boundary features \emph{causal bubbles}, a concept first introduced in the context of low regularity Lorentzian geometry \cite{ChruscielGrant}. Causal bubbles represent a breakdown of smooth causality where the causal past of a point is much larger than its timelike past. The notion of causal bubbles is discussed in detail in Sect.\,\ref{SS:CAUSALBUBBLES}.
\begin{corollary}[Causal bubbles] \label{cor:causalBubbles}
    Let $(\soluDomainClassical,\gfour)$ denote the MGHD constructed in Theorem \ref{thm:main:MGHD} and let $\gfourextend$ denote the continuous extension to the closure $\textnormal{cl}(\soluDomainClassical)$. Then any point $p \in \singularBoundary$ has a causal bubble with respect to the extended acoustical metric $\gfourextend$.
\end{corollary}

We highlight the following important fact: point \eqref{PT:NONUNIQUEINTEGRALCURVES} from Theorem \ref{thm:main:MGHD}, as well as the causal bubbles, imply that $\singularBoundary$ is acoustically $\emph{spacelike}$ from the point of view of its \emph{extrinsic} geometry. This is in contrast to point \eqref{pt:classicalProperties:boundaryIsNull} from Theorem \ref{thm:main:MGHD}, which states that $\singularBoundary$ is null with respect to the extended metric $\gfourextend$, i.e., with respect to its \emph{intrinsic} geometry, it is null. This surprisingly subtle fact was first discussed in Chapter 15 of Christodoulou's monograph \cite{Christodoulou:shockFormation} and the first and fourth author's work \cite{AbbresciaSpeck} in the context of $3D$ compressible fluid flow. We elaborate on this point in the much simpler context of our model problem in Sect.\,\ref{SS:CAUSALBUBBLES}.

We close this subsection by stating that, with minor additional effort, the results of Theorem \ref{thm:main:MGHD} could be proved to be \emph{stable}. That is, small perturbations of the initial data \eqref{E:QNLW:DATA:INTRO} give rise to solutions whose classical evolution ends due to a formation of shocks, and that the MGHD of this perturbed data satisfies the same quantitative (e.g. regularity) and qualitative (e.g. uniqueness) results of Theorem \ref{thm:main:MGHD}.

\subsection{Unique entropy and shock developed weak solutions} \label{SS:GLOBALWEAK}

We now discuss our second main result, which is roughly the construction of the unique global in spacetime entropy weak solutions  to \eqref{E:QNLW:INTRO}--\eqref{E:QNLW:DATA:INTRO}. To this end, we write our second order equation \eqref{E:QNLW:INTRO} in divergence form:
\begin{align} \begin{split}
    0 & = - \p_t(\p_t \solu - 2 \p_x \solu) - (2 + \p_t \solu - 2 \p_x \solu) \p_x (\p_t \solu - 2 \p_x \solu) \\
    & = - \p_t(\p_t \solu - 2 \p_x \solu) - \frac{1}{2}\p_x \left\{\left(2 + \p_t \solu - 2 \p_x \solu\right)^2\right\} \\
    & = -\p_t \soluBurgers - \frac{1}{2} \p_x (2+\soluBurgers)^2,\label{E:QNLWINDIVERGENCEFORM}
    \end{split}
\end{align}
where the last identity follows from \eqref{E:DEFINITIONOFPSI}. In particular, it follows that $\soluBurgers$ is a solution to Burgers' equation in divergence form. There is a well established uniqueness theory for entropy solutions to scalar conservation laws, which we now adapt to \eqref{E:QNLWINDIVERGENCEFORM}.

\subsubsection{Definitions of the classes of weak solutions} \label{SSS:WEAKSOLUTIONCLASSESDEFS}

In order to properly state our results, we first provide precise definitions of weak solutions as well as the class of entropy solutions for which our results apply to.

Our definition of weak solutions is motivated from Bressan's famous notes \cite{Bressan2013}.\footnote{Technically, solutions of the Cauchy problem in Definition 3 of \cite{Bressan2013} consider weak solutions defined on $[0,T]\times\R$. We have modified the definition to allow for more general domains.} It is clear from the definition below that the $\soluClassical$ constructed in Theorem\,\ref{thm:main:MGHD} is a weak solutions to \eqref{E:QNLWINDIVERGENCEFORM} on the MGHD $\soluDomainClassical$.

\begin{definition}[Weak solutions to \eqref{E:QNLWINDIVERGENCEFORM}] \label{D:WEAKSOLUTIONS}
Let $\Omega \subset \R^{1+1}$ be an open set. Then we say $\Phi \in W^{1,1}_{\textnormal{loc}}(\Omega)$ is a \emph{weak solution} to \eqref{E:QNLWINDIVERGENCEFORM} if, for any $\varphi \in C_c^1(\Omega)$, it follows that 
\begin{align} \label{E:WEAKSOLUTIONIDENTITY}
    0 = \int_{\Omega} \left\{ (\p_t \Phi - 2 \p_x \Phi) \p_t \varphi(t,x) + \frac{1}{2} (2 +\p_t \Phi - 2\p_x \Phi)^2 \p_x \varphi(t,x)\right\} \, \mathrm{d} t \mathrm{d} x.
\end{align}

Let $\Omega \subseteq \R^{1+1}$ be open and satisfy $\{t=0\} \times \R \subset \Omega$. Then we say $\Phi \in W^{1,1}_{\textnormal{loc}}(\Omega)$ is a \emph{weak solution to the Cauchy problem for} \eqref{E:QNLWINDIVERGENCEFORM} if, for any $\varphi \in C_c^1(\Omega)$, it follows that 
\begin{align}
    \begin{split} \label{E:WEAKSOLUTIONTOIVPIDENTITY}
        0 & = \int_{\Omega\cap\{t\ge 0\}} \left\{ (\p_t \Phi - 2 \p_x \Phi) \p_t \varphi(t,x) + \frac{1}{2} (2 + \p_t \Phi - 2\p_x \Phi)^2 \p_x \varphi(t,x)\right\} \, \mathrm{d} t \mathrm{d} x \\
        & \ \ + \int_{\Omega \cap \{t=0\}} (\p_t \Phi - 2 \p_x \Phi)(0,x) \varphi(0,x) \, \mathrm{d} x.
    \end{split}
\end{align}

\end{definition}

Uniqueness of weak solutions to hyperbolic conservation laws, much like the notion of uniqueness to MGHD's, is quite subtle. For the vast majority of the literature (see \cite{Bressan2013,cDafermos} and the references therein), uniqueness is typically discussed \emph{within a class of solutions satisfying certain properties}. For example, the well known Lax entropy conditions apply only to the class of weak solutions for which there is a curve (or hypersurface in multi-$D$) of discontinuity and a smooth solution on either side. 

In this paper, we study two classes of weak solutions to \eqref{E:QNLWINDIVERGENCEFORM}. The first is a class of Oleinik entropy solutions, for which there is a well established uniqueness theory for scalar conservation laws. 

\begin{definition}[Class of Oleinik entropy solutions] \label{D:CLASSOFENTROPYSOLUTIONS} We say $\Phi\in W^{1,\infty}(\Omega)$ is an Oleinik entropy solution, or just entropy solution for short, to \eqref{E:QNLWINDIVERGENCEFORM} on $\Omega \subseteq [0,\infty)\times \R$ if it is a weak solution as defined in Definition \ref{D:WEAKSOLUTIONS} and there exists a $C$ such that 
\begin{align}
    \frac{(\p_t \Phi - 2 \p_x \Phi)(t,x + z) - (\p_t \Phi - 2 \p_x \Phi)(t,x)}{z} \le C \left(1 + \frac{1}{t}\right) \label{E:OLEINIKCONDITION}
\end{align}
for almost every $x,z \in \R$ with $z > 0$. 

\end{definition}

\begin{remark}[Possible irregularity of Oleinik entropy solutions]
    It follows from \eqref{E:OLEINIKCONDITION} that the derivatives of these entropy solutions may be right-sided-Lipschitz while experiencing discontinuous jumps from the left.
\end{remark}

The other class of weak solutions we study are solutions to the \emph{shock development problem} (SDP). In Lax's well known survey article \cite{LaxFormationAndDecay72}, he gave a basic, qualitative description of the SDP:
\begin{quote}
  What happens after continuous solutions cease to exist? After all, the world does not come to an end. For an answer, we turn to experiments with compressible fluids: these clearly show the appearance of discontinuities in solutions.  
\end{quote}
In other words, the SDP describes the \emph{transition} from a smooth classical solution developing infinite second derivatives on $\initialSingularity$, to a weak solution with a single curve of discontinuity $\shockwave$ emanating from $\initialSingularity$ across which the Rankine--Hugoniot jump conditions and Lax entropy conditions hold. We now provide the precise definition of the shock development problem for our model.

\begin{definition}[The shock development problem for \eqref{E:QNLWINDIVERGENCEFORM}] \label{D:SHOCKDEVELOPMENT}
Let $\soluDomainClassical$ be the unique MGHD for the classical solutions $\soluClassical, \ \soluBurgersClassical:= \uLunit \soluClassical$ of \eqref{E:QNLW:INTRO}--\eqref{E:QNLW:DATA:INTRO} given by Theorem\,\ref{thm:main:MGHD}. In particular, the future boundary of $\soluDomainClassical$ is given by the initial singularity $\initialSingularity$, the singular boundary $\singularBoundary$, and the Cauchy horizon $\CauchyHorizon$. Let $\gfourextend$ denote the continuous extension to $\textnormal{cl}(\soluDomainClassical)$.

Then the shock development problem is to find a curve $\shockwave$ parametrized by $\{(t,k(t)) \ | \ t \in [1,T)\}$ referred to as \emph{the shock curve}, such that:
    \begin{enumerate}
        \item $\shockwave$ lies to the causal past of the singular boundary $\singularBoundary$. \label{pt:SDP:shockwaveInPastOfSingularBoundary}
        \item $\shockwave$ emanates from the initial singularity $\initialSingularity = (t=1,x=0) = (1,k(1))$. \label{pt:SDP:emanateFromInitialSingularity}
        \item $\shockwave$ is asymptotically $\gfour$-null at $\initialSingularity$ in the sense that it shares the same tangent vector to $\textnormal{cl}(\singularBoundary)$ at the initial singularity $\initialSingularity$. \label{pt:SDP:asymptoticallyNull}
        \item There is a weak solution $\soluSDP$ in the sense of Definition \ref{D:WEAKSOLUTIONS} with weak $\uLunit$ derivative $\soluBurgersSDP:= \partial_t \soluSDP - 2 \partial_x \soluSDP$ defined 
         on a region $\soluDomainSDP$ bounded in the past by $\CauchyHorizon$ and $\shockwave$ such that:
            \begin{enumerate} 
                \item On $\CauchyHorizon$, $(\soluBurgersSDP,\soluSDP)$ agrees with the smooth extension of $(\soluBurgersClassical,\soluClassical)$ to the Cauchy horizon.
                \item \label{E:SHOCKDEVELOPMENTEXTENSIONTOK} 
                To the left of $\shockwave$, $\soluBurgersSDP(t,x)$ is smoothly extendable up to $\shockwave$ and, by abuse of notation denoting the extended value along $\shockwave$ also by $\soluBurgersSDP$, this extension is uniquely determined by the \emph{Rankine--Hugoniot jump conditions}
                    \begin{align}
                        k'(t) & = \frac{\frac{1}{2} \big(2+\soluBurgersSDP(t,k(t))\big)^2 - \frac{1}{2}\big(2+\soluBurgersClassical(t,k(t))\big)^2}{\soluBurgersSDP(t,k(t)) - \soluBurgersClassical(t,k(t))} \notag \\
                        & = \frac{1}{2}\left\{ \big(2+\soluBurgersSDP(t,k(t))\big) + \big(2+\soluBurgersClassical(t,k(t))\big)\right\}.\label{E:RANKINEHUGONIOT}
                    \end{align}
            \end{enumerate} \label{pt:SDP:RankineHugoniot}
        \item Let $\gfour(\soluBurgersSDP)$ denote the acoustical metric in $\soluDomainSDP$ extended to the left of $\shockwave$ as in point \eqref{E:SHOCKDEVELOPMENTEXTENSIONTOK}. Then, with respect to $\gfour(\soluBurgersClassical)$, the shock curve is \emph{spacelike}. With respect to $\gfourextend(\soluBurgersSDP)$, the shock curve is \emph{timelike}.  This is called  \emph{the geometric determinism condition.}  \label{pt:SDP:supersonic}
    \end{enumerate}
\end{definition}

To see the ``transition'' mentioned before Definition \ref{D:SHOCKDEVELOPMENT}, let $J^-(\CauchyHorizon\cup\shockwave)$ represent the causal past of $\CauchyHorizon\cup\shockwave$ (which is a subset of the MGHD $\soluDomainClassical$), see Sect. \ref{SS:CAUSALBUBBLES}. Then the ``transition'' is given by the following weak solution:

\begin{align} 
    \begin{cases} \label{E:TRANSITIONSOLUTIONFORSDP}
        (\soluBurgersClassical,\soluClassical) & \text{on $J^-(\CauchyHorizon\cup\shockwave)$} \\
        (\soluBurgersSDP,\soluSDP) & \text{on $\soluDomainSDP$}.
    \end{cases}
\end{align}
Given \eqref{E:TRANSITIONSOLUTIONFORSDP}, one may view the geometric determinism condition as the following statement: the shock curve is \emph{supersonic} with respect to the solution to its causal past and \emph{subsonic} with respect to the solution on its causal future.

\begin{remark}[``Global'' shock development] \label{R:NOGLOBALDSHOCKDEVELOPMENT}
It will be clear from Theorem \ref{T:UNIQUEWEAKSOLUTIONS} below that our global unique entropy solution is a solution of the shock development problem with $T$ replaced by $\infty$. However, we do not make an effort to establish a notion of ``global'' solutions to the SDP because it is not clear the extent to which they exist for large times in multiple spatial dimensions. In particular, it is conjectured that singularities could form dynamically in finite time along shock hypersurfaces in multi-$D$, see Appendix \ref{S:PROBLEMSANDPERSPECTIVE} for a list of outstanding open problems. To date, the only global-in-time result in multi-$D$ concerning weak solutions to a hyperbolic PDE \emph{system} without symmetry assumptions is the recent work of Ginsberg--Rodnianski \cite{ginsberg2024stability}. The class of weak solutions studied in \cite{ginsberg2024stability} are those which contain two nearly spherical shock fronts dispersing away from one another at a precise rate. The PDEs in question there agree with the $3D$ irrotational isentropic Euler equations in the region outside of the outer-most-shock.

\end{remark}

\begin{remark}[The shock development problem in $1D$ or under symmetry reductions]
As we will see in the proof of Theorem\,\ref{T:UNIQUEWEAKSOLUTIONS}, we will prove the Lorentzian geometric notions of the SDP (e.g., $\shockwave$ is asymptotically $\gfour$-null at the initial singularity $\initialSingularity$, the determinism condition) as a \emph{consequence} of the already constructed Oleinik entropy solutions as stated in Remark \ref{R:NOGLOBALDSHOCKDEVELOPMENT}. That is, for our $1D$ model. Indeed, the reader will not find any mention of Lorentzian metrics in the first rigorous proof of the SDP by Lebaud \cite{lebaud1994description} for the $1D$ plane symmetric $p$-system. The same can be said for the generalization of Lebaud's work to  $3D$ spherically symmetric compressible Euler\cite{yin2004formation}, as well as $2D$ Euler in a symmetry class that allows for vorticity \cite{buckmaster2022simultaneous}. 

The reader might ask why, then, was the SDP formulated so geometrically in Definition \ref{D:SHOCKDEVELOPMENT}. The reason is that the Lorentzian geometry of the acoustical metric and its characteristics \emph{were fundamental in the only known SDP result in multiple dimensions without symmetry}: Christodoulou's 2019 breakthrough monograph on the \emph{restricted} shock development problem \cite{Christodoulou:shockDevelopment}. Our definition of the SDP is meant to emulate the one in \cite{Christodoulou:shockDevelopment} through the lens of our model. We discuss \cite{Christodoulou:shockDevelopment} in Sect. \ref{SSS:RESTRICTEDSHOCKDEVELOPMENT}. 

\end{remark}

\subsubsection{The second main result} \label{SS:MAINRESULTFORWEAKSOLUTIONS}

The precise statement of our main result for  entropy solutions is given as follows.

\begin{theorem}[Unique global entropy and shock developed weak solutions] \label{T:UNIQUEWEAKSOLUTIONS}
The Cauchy problem for \eqref{E:QNLWINDIVERGENCEFORM} has the following properties:

\begin{enumerate}
\item{} [Global existence and uniqueness] There is a weak solution $\soluWeak$ to \eqref{E:QNLWINDIVERGENCEFORM} on $\soluDomainWeak =\mathbb{R}^{1+1}$ that is unique within the class of Oleinik entropy solutions.  \label{pt:weakSolution:existenceUniqueness}
\item{} 
[$\soluWeak$ is not an extension of $\soluClassical$] Let $(\soluClassical,\soluDomainClassical)$ be the unique classical solution and MGHD of the Cauchy problem \eqref{E:QNLW:INTRO}--\eqref{E:QNLW:DATA:INTRO} given by Theorem \ref{thm:main:MGHD}. Then there is a point $p\in\soluDomainClassical$ at which the weak and classical solutions differ, i.e. $\soluWeak(p)\not=\soluClassical(p)$.
\label{pt:weakDoesntExtendClassical}
\item{} [Characterisation and global hyperbolicity of domain of agreement]
There is an open set $\soluDomainAgree\subsetneq\soluDomainClassical$ on which the two solutions agree, i.e. on $\soluDomainAgree$, $\soluClassical=\soluWeak$. Furthermore, $\soluDomainAgree$ is a globally hyperbolic development of $\hsfcData$ with respect to the initial value problem \eqref{E:QNLW:INTRO}--\eqref{E:QNLW:DATA:INTRO}, and $\hsfcData\subsetneq\soluDomainAgree\subsetneq\soluDomainClassical\subsetneq\soluDomainWeak$.  \label{pt:weakSolution:characterisationOfDomain}
\item{} [Shock development] The unique entropy solution $\soluWeak$ constructed above is also in the class of solutions to the shock development problem as in Definition \ref{D:SHOCKDEVELOPMENT}. In particular, the SDP is solved on $\soluDomainSDP := \big(\R^{1+1}\big) \setminus \soluDomainAgree$ by $(\soluBurgersSDP,\soluSDP) := \big(\soluBurgersWeak|_{\soluDomainSDP},\soluWeak|_{\soluDomainSDP}\big)$. \label{pt:weakSolution:shockDevelopment}
\item{}[Geometric determinism of the shock]  The geometric determinism condition of the shock development problem (see Definition \ref{D:SHOCKDEVELOPMENT} part \eqref{pt:SDP:supersonic})   is equivalent to the \emph{Lax entropy conditions}, that is, 
            \begin{align}
                2 + \soluBurgersClassical(t,k(t)) < k'(t) < 2 + \soluBurgersSDP(t,k(t)). \label{E:LAXENTROPY}
            \end{align} \label{pt:weakSolution:geometricDeterminism}
\item{} [Weak singularities along the Cauchy horizon]
The weak solution $\soluWeak$ extends the classical solution $\soluClassical$ in a $C^1$ manner across $\CauchyHorizon$. However, on $\soluDomainSDP$, the first derivatives of $\soluWeak$ only extend to $\CauchyHorizon$ in $C^{0,1/2}$. \label{pt:weakSolution:weakSingularityAcrossCauchyHorizon}
\end{enumerate}
\end{theorem}

\begin{figure}
\includegraphics[]{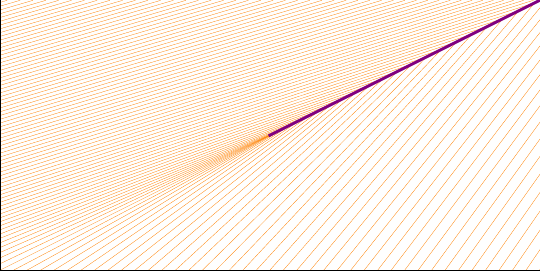}
\caption{An illustration of the domain for the solution provided by Theorem \ref{T:UNIQUEWEAKSOLUTIONS} which is both unique in the Oleinik entropy class and a solution to the shock development problem. The outgoing characteristics are given in solid orange, while the shock curve is given in purple.}
\label{fig:BurgersBlowUp}
\end{figure}

\subsubsection{Christodoulou's $3D$ restricted shock development problem} \label{SSS:RESTRICTEDSHOCKDEVELOPMENT}

We now describe Christodoulou's breakthrough solution \cite{Christodoulou:shockDevelopment} to the restricted shock development problem in more detail. He studied solutions to a class of $3D$ quasilinear wave equations of the form
\begin{align} \label{E:CHRISTODOULOUWAVEEQUATIONS}
    (\mathbf{g}^{-1})^{\alpha \beta}(\pmb{\partial} \Phi)
    \partial_{\alpha} \partial_{\beta} \Phi
    & = 0.
\end{align}
In \eqref{E:CHRISTODOULOUWAVEEQUATIONS}, 
$\mathbf{g}$ is a Lorentzian metric (a quadratic form of signature $(-,+,+,+)$) 
whose standard coordinate components $\mathbf{g}_{\alpha \beta}$ are functions of the spacetime gradient of $\Phi$, denoted by $\pmb{\partial} \Phi$.
The particular equations studied in \cite{Christodoulou:shockDevelopment} came from irrotational and isentropic compressible fluid mechanics, either relativistic or non-relativistic. More precisely, for \emph{classical} solutions -- which was not the class of solutions studied in \cite{Christodoulou:shockDevelopment} -- $\Phi$ can be interpreted as a potential function, and the density and velocity can be computed in terms of $\pmb{\partial} \Phi$; the detailed expressions as well as the precise structure of $\mathbf{g}_{\alpha \beta}$ depend on the fluid equation of state.
For the below discussion, it is crucial to appreciate that the fluid equations can be written in divergence form and hence for the wave equations studied in \cite{Christodoulou:shockDevelopment},
equation \eqref{E:CHRISTODOULOUWAVEEQUATIONS} can be written in divergence form as follows:
\begin{align}  \label{E:DIVERGENCECHRISTODOULOUWAVEEQUATIONS}   
    \partial_{\alpha}
    \left\lbrace
        \mathcal{N}(\pmb{\partial} \Phi)
        (m^{-1})^{\alpha \beta}\partial_{\beta} \Phi
    \right\rbrace
    & = 0,
\end{align}
where $m = \mbox{diag}(-1,1,1,1)$ is the Minkowski metric and the precise form of the nonlinearity $\mathcal{N}(\pmb{\partial} \Phi)$ depends on the fluid equation of state, see \cite[Eq. (1.87)]{Christodoulou:shockDevelopment}.

Christodoulou's initial data for solving the SDP were a portion of an MGHD that can be split into four key regions: 
\begin{enumerate}
\item A region of classical existence.
\item The singular boundary, that is, a manifold-with-boundary along which 
    $\pmb{\partial}^2 \Phi$ blows up, with $\pmb{\partial} \Phi$ and $\Phi$ remaining bounded. 
    Physically, the boundedness of $\pmb{\partial} \Phi$ implies the boundedness of the velocity and density.
\item The past boundary of the singular boundary, which is a co-dimension $2$ spacelike\footnote{Note that it make sense to talk about $\mathbf{g}(\pmb{\partial} \Phi)$ on the singular boundary because $\pmb{\partial} \Phi$ remains bounded up to it.} (with respect to $\mathbf{g})$ sub-manifold
    serving as a higher-dimensional analog of the cusp point in Fig.\,\ref{F:BOTHSOLUTIONSINONEPICTURE}.
    This set is referred to as the \emph{crease} in \cite{AbbresciaSpeck} and has been alternatively referred to as the ``first singularity'' or ``pre-shock'' by other authors. In \cite{Christodoulou:shockDevelopment}, the crease was assumed to have the topology of $\mathbb{S}^2$, though that was not of fundamental importance for the analysis. In contrast, the fact that the set is a spacelike co-dimension $2$ submanifold was fundamentally important.
\item The Cauchy horizon, which is a $\mathbf{g}$-null hypersurface-with-boundary emanating from the crease such that the solution remains smooth up to it, except at the crease (where $\pmb{\partial}^2 \Phi$ blows up). The Cauchy horizon therefore lies in the domain of influence of the singularity, but the singularity does not propagate along it.
\end{enumerate}

One might ask if initial data for the SDP can arise from $C^{\infty}$ initial data prescribed on $\lbrace t = 0 \rbrace$. The answer is essentially ``yes,'' with a few caveats. Specifically, in \cite{Christodoulou:shockFormation}, Christodoulou studied open sets of smooth initial data given at $\lbrace t = 0 \rbrace$ for  equation \eqref{E:CHRISTODOULOUWAVEEQUATIONS} in $3$ space dimensions, and he derived a large -- though not fully explicit -- portion of the classical MGHD, including a portion of the crease and singular boundary. The first constructive description of an entire connected component of the crease and an $O(1)$ portion of the singular boundary
emanating from it was derived in \cite{AbbresciaSpeck}, where the authors treated the case of $3$ space dimensions for solutions with vorticity and entropy. The first constructive description of a portion of an MGHD that also included a portion of a Cauchy horizon was derived in \cite{shkoller2024geometry}. More precisely, for adiabatic equations of state in $2$ space dimensions, the authors studied isentropic solutions with size $O(1/\epsilon)$ gradients for the initial velocity and density. For $\epsilon$ sufficiently small, they derived an $O(\epsilon)$-portion of the crease and singular boundary, the portion that is contained within an $O(\epsilon)-$sized spacetime slab whose top and bottom boundaries are flat constant-time hypersurfaces. We also mention the forthcoming work \cite{AbbresciaSpeck2}, which in $3$ space dimensions for solutions with vorticity and entropy will complement the results of \cite{AbbresciaSpeck} by providing an $O(1)$ portion of a Cauchy horizon emanating from an entire connected component of the crease.

We now informally describe what is a solution to the SDP. The starting point is the initial data as described above and the divergence form \eqref{E:DIVERGENCECHRISTODOULOUWAVEEQUATIONS} of the equation. The fruitful notion of a  solution, adopted here from \cite{Christodoulou:shockDevelopment}, is a higher-dimensional analog of Definition~\ref{D:SHOCKDEVELOPMENT}. More precisely, a (local) solution to SDP is a portion of a shock hypersurface $\mathcal{K}$ emanating from the crease and a weak solution to equation \eqref{E:DIVERGENCECHRISTODOULOUWAVEEQUATIONS} that is defined in a neighborhood of the crease and that is smooth on either side of $\mathcal{K}$, where appropriate limits of $\Phi$ up to 
$\mathcal{K}$ from either side exist. The definition of a weak solution from \cite{Christodoulou:shockDevelopment} is essentially a higher-dimensional analog\footnote{More precisely, in \cite{Christodoulou:shockDevelopment}, Christodoulou's notion of a weak solution does not involve test functions. Instead, he integrates the divergence-form equation \eqref{E:DIVERGENCECHRISTODOULOUWAVEEQUATIONS} over an arbitrary spacetime domain $\mathcal{V}$ and then posits that Stokes' theorem can be applied, i.e., he replaces the spacetime integral over $\mathcal{V}$ with the boundary integrals that one would formally obtain from Stokes' theorem if $\Phi$ were $C^2$. This yields an identity for boundary integrals. By appropriately adapting the domains to $\mathcal{K}$ and assuming that $\Phi$ is smooth up to $\mathcal{K}$ on either side of it (i.e., piecewise smooth), one derives the jump condition \eqref{E:RANKINEHUGONIOTFORRETRICTEDSHOCK}; see \cite[Section~1.4]{Christodoulou:shockDevelopment} 
for more details. This same jump condition could alternatively be derived by imposing a weak formulation of equation \eqref{E:DIVERGENCECHRISTODOULOUWAVEEQUATIONS}  using test functions, as in Definition~\ref{D:WEAKSOLUTIONS}.} 
of the one given in Definition~\ref{D:WEAKSOLUTIONS}. As in the one-space-dimensional case, this leads to nonlinear Rankine--Hugoniot jump conditions, specifically:
\begin{align} \label{E:RANKINEHUGONIOTFORRETRICTEDSHOCK}
    \xi_{\alpha} [[\mathcal{N}(\pmb{\partial} \Phi)(m^{-1})^{\alpha \beta}\partial_{\beta} \Phi]]
    & = 0, 
    \mbox{at all points in } \mathcal{K}
\end{align}
where $\xi$ is the co-vector that is co-normal to $\mathcal{K}$ and
$[[\mathcal{N}(\pmb{\partial} \Phi)(m^{-1})^{\alpha \beta}\partial_{\beta} \Phi]]$ 
denotes the jump across $\mathcal{K}$
in the nonlinearity from equation \eqref{E:DIVERGENCECHRISTODOULOUWAVEEQUATIONS}.
Moreover, as is described in \cite{Christodoulou:shockDevelopment}, $\Phi$ itself is continuous across $\mathcal{K}$, the derivatives of $\Phi$ in directions tangent to $\mathcal{K}$ should be continuous across $\mathcal{K}$, and $\Phi$ and its full spacetime gradient are continuous across the Cauchy horizon. 

For the well-posedness of SDP, one fundamentally needs \emph{determinism conditions}, which we now explain. The first determinism condition states that $\mathcal{K}$ should be spacelike with respect to the metric $\mathbf{g}(\pmb{\partial} \Phi)$ from equation \eqref{E:CHRISTODOULOUWAVEEQUATIONS} when $\Phi$ is the past solution, before $\mathcal{K}$. From a practical perspective, this means that $\mathcal{K}$ sticks out into the classical MGHD, and thus the behavior of the solution to the past of $\mathcal{K}$ is determined by smooth data
given at $\lbrace t = 0 \rbrace$, in fact uniquely determined if the MGHD is unique! The second determinism condition states that
$\mathcal{K}$ should be timelike with respect to $\mathbf{g}(\pmb{\partial} \Phi)$ when $\Phi$ is the future solution, i.e., the solution after the jump across $\mathcal{K}$. 
This second determinism condition is used in the analysis to uniquely determine the jump across $\mathcal{K}$ and to therefore connect the future solution on $\mathcal{K}$ with the past one on 
$\mathcal{K}$.

We now emphasize the difference between the restricted shock development problem and the full shock development problem for compressible fluids. 
Moreover, as we mentioned above, $\Phi$ and its derivatives in directions tangent to $\mathcal{K}$ should be continuous across $\mathcal{K}$, and $\Phi$ and its full spacetime gradient are continuous across the Cauchy horizon. 
For the full compressible Euler equations in three space dimensions, instead of the scalar conservation law \eqref{E:DIVERGENCECHRISTODOULOUWAVEEQUATIONS}, one has 5 conservation laws, representing the conservation of mass, momentum (a three-vector), and energy. One can prove that for smooth solutions with vanishing vorticity and constant entropy, the compressible Euler equations are equivalent to equation \eqref{E:DIVERGENCECHRISTODOULOUWAVEEQUATIONS}. However, the equivalence of the two formulations does not generally hold for weak solutions. In particular, for piecewise smooth solutions that jump across a shock hypersurface $\mathcal{K}$, the jump conditions for the full compressible Euler equations -- which are obtained from 5 coupled conservation laws -- are not equivalent to \eqref{E:RANKINEHUGONIOTFORRETRICTEDSHOCK}. For example, the full compressible Euler jump conditions imply that the entropy $s$ should jump across $\mathcal{K}$, and that for small shocks, the entropy jump, denoted $\Delta s$, should scale like $(\Delta p)^3$, where $\Delta p$ is the jump in the fluid pressure across $\mathcal{K}$; see \cite[Section~1.4]{Christodoulou:shockDevelopment}. Therefore, even if the fluid solution is isentropic and irrotational in the region of classical existence, the weak solution will develop entropy (and, as it turns out, vorticity) to the future of the crease. This effect is not captured

It is also worth noting that in the context of fluid mechanics with vorticity and entropy, if the jump in fluid pressure is sufficiently small, then 
the aforementioned determinism conditions for $\mathcal{K}$ are equivalent to the entropy strictly increasing across $\mathcal{K}$, i.e., to the second law of thermodynamics;
see \cite[Equation (1.291) and (1.328)]{Christodoulou:shockDevelopment} and the discussion around it. However, if the jump in pressure is large, then there is no guarantee that the determinism conditions -- which seem essential for the well-posedness of SDP -- will cause the entropy to increase across $\mathcal{K}$ (in principle, it could decrease).

\subsection{Overview of the rest of the article}

Section \ref{s:Burgers} discusses solutions of Burgers' equation, which will be used as foundations for the proof of the main theorems in Section \ref{s:pfMainThm}. Appendix \ref{SecApp} provides a brief summary of some of the key causal concepts, which are common in the general relativity community but applicable more generally in the study of quasilinear wave equations. Appendix \ref{S:PROBLEMSANDPERSPECTIVE} surveys some open problems in the study of shocks and singularity formation.

\subsection{Acknowledgements}

L. Abbrescia gratefully acknowledges support from a Travel Support for Mathematicians grant from the Simons Foundation.
J. Sbierski gratefully acknowledges support through the Royal Society University Research Fellowship URF\textbackslash R1\textbackslash 211216. J. Speck gratefully acknowledges support from NSF grants DMS-2349575 and DMS-2054184. We are grateful to ICMS for funding the conference ``Singularity formation in General Relativity and Dispersive PDE'', where many of the core ideas in this paper were first developed.


\section{Relevant points from Burgers' equation}
\label{s:Burgers} 

We find it convenient to set up the proofs of Theorems \ref{thm:main:MGHD} and \ref{T:UNIQUEWEAKSOLUTIONS} in the context of Burgers' equation $\Lunit \soluBurgers = \p_t \soluBurgers + (2+\Psi) \p_x \soluBurgers = 0$. This is because expressing the Cauchy problem for the quasilinear wave equation \eqref{E:QNLW:INTRO}--\eqref{E:QNLW:DATA:INTRO} as a system of first order equations produces a fully decoupled Burgers' equation \eqref{E:FirstOrderBurgers} in the direction of $\Lunit$, see \eqref{E:DOUBLENULLFRAME} and \eqref{E:QNLWINDIVERGENCEFORM}: 
\begin{subequations}
\label{E:FIRSTORDERSYSTEMCAUCHYPROBLEM}
    \begin{align}
        \Lunit \Psi & = 0, &  \Psi(0,x) & := \dataBurgers(x) = - \arctan(x), 
        \label{E:FirstOrderBurgers}\\
        \uLunit \Phi & = \Psi,   & \Phi(0,x) & :=  0. 
        \label{E:FirstOrderWavePartshockdevelopmentp}
    \end{align}
\end{subequations}
Consequently, the genuinely non-linear nature of the outgoing characteristic $\Lunit$ is the only catalyst for the singularities in our main results. After solving Burgers, we simply treat it as a source term to the linear transport equation \eqref{E:FirstOrderWavePartshockdevelopmentp} that can be explicitly solved. We stress, however, that \emph{there is no causal structure associated to Burgers' equation alone} because the ingoing characteristics $\uLunit$ are not featured in \eqref{E:FirstOrderBurgers}. In particular, even though \eqref{E:FirstOrderBurgers} is equivalent to Burgers equation, without the second characteristic speed in the system given by $\uLunit$ in \eqref{E:FirstOrderWavePartshockdevelopmentp}, \emph{there is no notion of global hyperbolicity} for Burgers equation. In fact, it is easy to prove that classical extensions past the initial singularity for \eqref{E:FirstOrderBurgers} \emph{without} considering \eqref{E:FirstOrderWavePartshockdevelopmentp} are non-unique.

\subsection{Burgers equation on the relevant spacetime regions for the MGHD and shock development problem}
\label{ss:BurgersExample}
We first introduce the spacetime regions mentioned in our main theorems and then analyze the behavior of the solution to Burgers equation on them. For the reader's convenience, we reproduce Fig.\,\ref{F:BOTHSOLUTIONSINONEPICTURE} here. 

\begin{figure}[h]
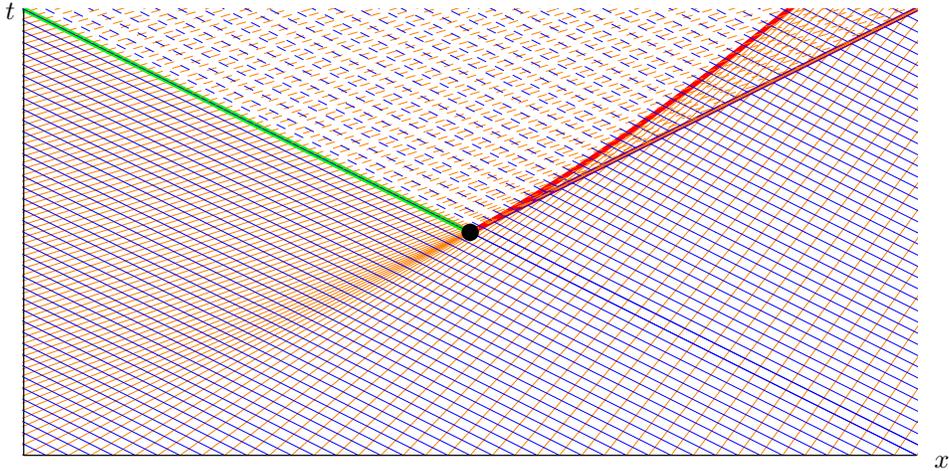

    \begin{overpic}[scale=1.3, grid = false, tics=5, trim=-.5cm 0cm -1cm -.5cm, clip]{three_regions_with_shear.pdf}
        \put (3,42) {$t$}
        \put (92,-1) {$x$}
    \end{overpic}
    \caption{A figure of the ingoing and outgoing characteristics on the relevant spacetime regions denoted by the solid blue and orange lines, respectively.}
    \label{F:REPRODUCEDBOTHSOLUTIONSINONEPICTURE}
\end{figure}

The shapes of the regions in Figure\,\ref{F:REPRODUCEDBOTHSOLUTIONSINONEPICTURE} are governed by the integral curves of $\Lunit$, namely, the outgoing characteristics. For $t \in [0,\infty)$ and $x_0 \in\mathbb{R}$, we define
	\begin{align}
			\charcurveL_{x_0}(t)  & := (t,x_0 + t(2+ \dataBurgers(x_0))). \label{E:OUTGOINGCHARCURVEFORMULA}
        \end{align}
It will be shown in Lemma\,\ref{L:BURGERSSOLUTIONSMOOTHONMGHD} that these are in fact the integral curves of $\Lunit$ on the MGHD.

\begin{definition}[The key spacetime regions] \label{D:KEYSPACETIMEREGIONS}
For $x_0\in[0,\infty)$, we denote the $t$ and $x$ coordinates of $\charcurveL_{x_0}(t)$ evaluated at $t = -\frac{1}{\dataBurgers'(x_0)}$ to be
\begin{subequations}
\begin{align}
\tSingularBoundaryFromData(x_0) & := - \frac{1}{\dataBurgers'(x_0)} = 1+x_0^2 ,\label{E:TCOORDINATEOFSINGULARSETFROMDATA} \\
\xSingularBoundaryFromData(x_0) & := (2-\arctan(x_0))(1+x_0^2) +x_0 \label{E:XCOORDINATEOFSINGULARSETFROMDATA}.
\end{align}
\end{subequations}
For $t\geq 1$, define the functions
\begin{subequations}
\begin{align}
\xSingularBoundary(t) & := \xSingularBoundaryFromData(\sqrt{t-1}) ,
\label{eq:positionOfSingularBoundary}\\
\xRegularBoundary(t) & := 4 - 2t
\end{align}
\end{subequations}

Define the preliminary sets
\begin{subequations}
\begin{align}
\BurgersDomainRight & := \bigcup_{x \in (0,\infty)} \left\{ \charcurveL_{x}(t) \ | \ t \in \left(-\infty,-\frac{1}{\dataBurgers'(x)}\right) \right\} , \label{E:BURGERSDOMAINRIGHT} \\
\BurgersDomainLeft & := \bigcup_{x \in (-\infty,0]} \left\{ \charcurveL_{x}(t) \ | \ t \in \left(-\infty,\frac{-4 + x}{-4 - \dataBurgers(x)}\right) \right\},  \label{E:BURGERSDOMAINLEFT} \\
\Omega_3 & := \bigcup_{x \in (0,\infty)}\left\{ \charcurveL_{x}(t) \ | \ t \in \left(-\infty,-\frac{x}{\dataBurgers(x)}\right) \right\},
\end{align}
\end{subequations}
which denote spacetime sets foliated by outgoing characteristics that terminate at the singular boundary, the the Cauchy horizon, and the shock, respectively. Denote further the sets\footnote{It will be shown later that there is a weak solution on $\soluDomainWeak$, there is a classical solution on $\soluDomainClassical$, and the two solutions agree on $\soluDomainAgree$.}  
\begin{subequations}
\begin{align}
\soluDomainWeak & = \Reals^{1+1} ,\\
\soluDomainClassical & = \BurgersDomainRight \cup \BurgersDomainLeft, \label{E:DEF:MGHDASASET}\\
\soluDomainAgree & =  \BurgersDomainLeft \cup \Omega_3,
\end{align}
\label{eq:soluDomains}
\end{subequations}
the $1$-dimensional sets,  
\begin{subequations}
\begin{align}
\singularBoundary & = \{(t,x): t>1, x=\xSingularBoundary(t) \},
\label{eq:def:singularBoundary}\\
\CauchyHorizon & = \{(t,x): t>1, x=\xRegularBoundary(t) \} , \label{eq:def:cauchyhorizon} \\
\shockwave & = \left\{(t,x): t>1, x = 2t \right\}, \label{eq:def:discontinuousshock}
\end{align}
\end{subequations}
and the point
\begin{align}
\initialSingularity ={}& (\tSingularBoundaryFromData(0),\xSingularBoundaryFromData(0)) = (t = 1,x = 2).
\end{align}
In Figure \ref{F:REPRODUCEDBOTHSOLUTIONSINONEPICTURE}, the \textbf{Cauchy horizon} $\CauchyHorizon$ is given by the green curve, the \textbf{singular boundary} $\singularBoundary$ is given by the red curve, the \textbf{shock} $\shockwave$ is given by the purple curve, and the \textbf{initial singularity} $\initialSingularity$ is given by the black point. The region $\soluDomainClassical$, which will ultimately be the MGHD, is represented by the region underneath $\CauchyHorizon\cup \initialSingularity\cup \singularBoundary$ and is where both the ingoing and outgoing characteristics are solid. The region $\soluDomainAgree$ is represented by the region underneath $\CauchyHorizon\cup \initialSingularity\cup \shockwave$. 
\end{definition}

The following lemma describes the classical solution of Burgers' equation on $\soluDomainClassical$.

\begin{lemma}[The Classical solution on $\soluDomainClassical$] \label{L:BURGERSSOLUTIONSMOOTHONMGHD}
    Let $\soluDomainClassical$ be as in \eqref{E:DEF:MGHDASASET}. Then $\charcurveL_{x_0}(t)$ are the characteristic curves of $\Lunit$ emanating from $(0,x_0) \in \{0\}\times\R$, respectively. That is, 
	\begin{align}
		\frac{\mathrm{d}}{\mathrm{d} t} \charcurveL_{x_0}^\alpha(t) = \Lunit^\alpha \circ \charcurveL_{x_0}(t), & &  \charcurveL_{x_0}(0) = (0,x_0) \label{E:DEFINITIONOFCHARACERISTICCURVES}
	\end{align}
and the following identities hold along them:
	\begin{align} \label{E:SOLUTIONCONSTANTONCHARACTERISTICS}
		\soluBurgers \circ \charcurveL_{x_0}(t) = \dataBurgers(x_0), & & \p_x \soluBurgers \circ \charcurveL_{x_0}(t) =  \frac{\dataBurgers'(x_0)}{1 + t\dataBurgers'(x_0)}.
    \end{align}

Consequently, the Cauchy problem \eqref{E:FirstOrderBurgers} has a unique solution $\soluBurgers \in C^1(\soluDomainClassical)$.
\end{lemma}

\begin{proof}
    Differentiating Burgers' equation $\Lunit \soluBurgers = 0$ shows that $\p_x \soluBurgers$ solves the following transport equation:
        \begin{align}
            \Lunit \p_x \soluBurgers = - (\p_x \soluBurgers)^2. \label{E:TRANSPORTEQUATIONFORONEDEROFBURGERS}
        \end{align}
    Identities \eqref{E:DEFINITIONOFCHARACERISTICCURVES}--\eqref{E:SOLUTIONCONSTANTONCHARACTERISTICS} follow from a standard application of the method of characteristics and \eqref{E:TRANSPORTEQUATIONFORONEDEROFBURGERS}.

    By \eqref{E:SOLUTIONCONSTANTONCHARACTERISTICS}, proving that the characteristics do not intersect in $\soluDomainClassical$ will imply that $\soluBurgers \in C^1(\soluDomainClassical)$. Since the characteristics within $\soluDomainClassical$ are straight line segments by \eqref{E:OUTGOINGCHARCURVEFORMULA}, if two of them intersect, then it must occur at a unique point. Fix $x_1,x_2 \in \R$ and assume for the sake of contradiction that $\charcurveL_{x_1}$ and $\charcurveL_{x_2}$ intersect. By \eqref{E:OUTGOINGCHARCURVEFORMULA}, this must occur at time
    \begin{align}
        t = - \frac{x_2 - x_1}{\dataBurgers(x_2) - \dataBurgers(x_1)}. \label{eq:intersectionTimeGeneral}
    \end{align}

    First, consider the case $0 <x_1 < x_2$, and note that $\dataBurgers'(x) = - \frac{1}{1+x^2}$ is non-zero and strictly increasing on $(0,\infty)$. This strict convexity implies 
    \begin{align}
        - \frac{x_2 - x_1}{\dataBurgers(x_2) - \dataBurgers(x_1)} > - \frac{1}{\dataBurgers'(x_1)},
    \end{align} which contradicts our choice of upper bound for $t$ in \eqref{E:BURGERSDOMAINRIGHT}. Now, consider the case $x_1 \le 0 < x_2$. Consider the three lines $\charcurveL_{x_1}$, $\CauchyHorizon$, and $\charcurveL_{x_2}$. From properties of the arctangent, we find, at $t=1$, the respective $x$ coordinates of these lines are $x_1+2-\arctan(x_1)\leq2$, $2$, and $x_2+2-\arctan(x_2)>2$. The line $\charcurveL_{x_1}$ intersects $\CauchyHorizon$ when $x_1+t(2+\dataBurgers(x_1))=4-2t$, which occurs at \begin{align}
        t=\frac{4-x_1}{4+\dataBurgers(x_1)}, 
    \end{align}
    which we chose as the upper bound for $t$ in \eqref{E:BURGERSDOMAINLEFT}. Observe that the difference in slopes between $\charcurveL_{x_1}$ and $\charcurveL_{x_2}$ is $(2+\dataBurgers(x_2))-(2+\dataBurgers(x_1))$ $=\dataBurgers(x_2)-\dataBurgers(x_1)<0$, so since $\charcurveL_{x_2}(1)$ lies to the right of $\charcurveL_{x_1}(1)$, their intersection occurs at $t>1$. Similarly, since $-2-(2+\dataBurgers(x_1))<0$, the lines $\charcurveL_{x_1}$ and $\CauchyHorizon$ intersect at $t\geq1$. Furthermore, since $-2<2+\dataBurgers(x_2)$, the line $\CauchyHorizon$ starts at $t=1$ further to the left and is sloped more steeply to the left than $\charcurveL_{x_2}$, so the line $\charcurveL_{x_1}$ intersects $\CauchyHorizon$ before intersecting $\charcurveL_{x_2}$. Finally, consider the case $x_1<x_2<0$. Within this argument, set $a_1=\arctan(x_1)$ and $a_2=\arctan(x_2)$. The difference between the times of intersection of $\charcurveL_{x_2}$ with $\charcurveL_{x_1}$ and $\CauchyHorizon$ is 
    $(x_2-x_1)/(a_1-a_2)-(4-x_2)/(4-a_2)$, which, when combined as a single fraction, has a positive denominator and numerator given by 
    $(x_2-x_1)(4-a_2)-(4-x_2)(a_1-a_2)$
    $=4x_2-4x_1-x_2a_1+x_1a_2-4a_1+a_1x_2+4a_2-x_2a_2$
    $=4(x_2-x_1) +4(a_2-a_1)+x_1a_2+a_1x_2$, 
    which is a sum of for positive terms, and hence positive; therefore, $\charcurveL_{x_2}$ and $\charcurveL_{x_1}$ do not intersect until after $\charcurveL_{x_2}$ has left $\BurgersDomainLeft$.  Since the characteristic curves never intersect in $\BurgersDomainClassicalOnly$, the solution $\soluBurgersClassical$ is $C^1$ and unique there.

\end{proof}

\begin{lemma}
\label{lem:topologyGeometryOfSingularity}
[The topology/geometry of the sets $\initialSingularity\cup\singularBoundary$ and the behavior of $\soluBurgers$ on them] 
\label{lem:singularBehaviourOfBurgersAtSingularBoundary} 
Let $\soluDomainClassical$ be the domain as in \eqref{E:DEF:MGHDASASET} and let $\soluBurgersClassical \in C^1(\soluDomainClassical)$ denote the solution to \eqref{E:FirstOrderBurgers} from Lemma \ref{L:BURGERSSOLUTIONSMOOTHONMGHD}. For $x_0 \in [0,\infty)$, consider the continuous extension of $\soluBurgersClassical$ to $\initialSingularity\cup\singularBoundary$ given by $\soluBurgersClassical(\tSingularBoundaryFromData(x_0),\xSingularBoundaryFromData(x_0)) := \dataBurgers(x_0)$, see \eqref{E:OUTGOINGCHARCURVEFORMULA}, \eqref{E:TCOORDINATEOFSINGULARSETFROMDATA}--\eqref{E:XCOORDINATEOFSINGULARSETFROMDATA}, and \eqref{E:SOLUTIONCONSTANTONCHARACTERISTICS}. Consider the corresponding continuously extended vectorfield $\Lunit = \partial_t + (2+\soluBurgersClassical)\partial_x$.
\begin{enumerate}
    \item The boundary of $\soluDomainClassical$ satisfies $\p\soluDomainClassical = \initialSingularity\cup\singularBoundary\cup\CauchyHorizon$. \label{pt:boundaryofMGHD}
    \item Let $\{(t,\xSingularBoundary(t)) : t \in [1,\infty)\}$ denote the parametrization of $\initialSingularity\cup \singularBoundary$ from \eqref{eq:def:singularBoundary}. Then $\xSingularBoundary'(t) = 2 + \soluBurgersClassical(t,\xSingularBoundary(t))$. That is, the extended vectorfield $\Lunit$ is tangent to $\initialSingularity\cup\singularBoundary$. \label{pt:tangentOfSingularBoundary}
    \item For $\bar{t}>1$, there is a unique $\bar{x}_0$ such that $\bar{t}=\tSingularBoundaryFromData(\bar{x}_0)$ $=1/\dataBurgers'(\bar{x}_0)$ . Let $\bar{x}=\xSingularBoundaryFromData(\bar{x}_0)$. For $t=\bar{t}$, $x>\bar{x}$ such that $x-\bar{x}$ is small, 
\begin{align}
\soluBurgersClassical(\bar{t},x)
-\soluBurgersClassical(\bar{t},\bar{x})
& = \dataBurgers'(\bar{x}_0) \sqrt{\frac{2 {|\dataBurgers'(\bar{x}_0)|}}{\dataBurgers''(\bar{x}_0)}} (x-\bar{x})^{\frac12}
        +O((x-\bar{x})^{3/4}) \label{eq:expansionNearSingularBoundary}
\end{align}  
\label{pt:expansionNearSingularBoundary}
    \item Fix $t = 1$ and $x = 2$, i.e., the spacetime coordinates of $\initialSingularity$. If $|x - 2|$ is sufficiently small, we have
        \begin{align}
            \soluBurgersClassical(1,x) - \soluBurgersClassical(1,2) = -(3(x - 2))^{1/3} + O((x-2)^{4/9}). \label{E:EXPANSIONNEARS}
        \end{align} \label{PT:EXPANSIONNEARS}
\end{enumerate}   
\end{lemma}

\begin{proof}
    Point \eqref{pt:boundaryofMGHD} follows from the definition of the set $\soluDomainClassical$, see \eqref{E:DEF:MGHDASASET}. Point \eqref{pt:tangentOfSingularBoundary} follows from a direct calculation. 

    We now prove point \eqref{pt:expansionNearSingularBoundary}. Let $x_0$ be such  that the characteristic through $(t,x)$ goes through $(0,x_0)$.  Since, $(\bar{t},\bar{x})=(-\frac{1}{\dataBurgers'(\bar{x}_0)},\bar{x})$, it follows from the explicit parametrization of the characteristics \eqref{E:OUTGOINGCHARCURVEFORMULA} that
    \begin{align}
        x-\bar{x}
        ={}& -\frac{\dataBurgers(x_0)-\dataBurgers(\bar{x}_0)}{\dataBurgers'(\bar{x}_0)}
        +(x_0-\bar{x}_0) \notag \\
        ={}& -\frac{1}{\dataBurgers'(\bar{x}_0)}
        \left(\dataBurgers(x_0)-\dataBurgers(\bar{x}_0)-\dataBurgers'(\bar{x}_0)(x-x_0)\right) \notag \\
        ={}& -\frac{1}{\dataBurgers'(\bar{x}_0)}  \frac{1}{2} \dataBurgers''(\bar{x}_0)(x_0-\bar{x}_0)^{2} +O(x_0-\bar{x}_0)^3) . 
        \label{eq:intermediateExpansionNearSingularBoundary}
    \end{align}
    Solving for $x_0-\bar{x}_0$, one finds
    \begin{align}
        x_0-\bar{x}_0  \label{E:XCOORDINATENEARB}
        & = \sqrt{\frac{2|\dataBurgers'(x_0)|}{\dataBurgers''(\bar{x}_0)}} (x-\bar{x})^\frac12
        +O((x-\bar{x})^\frac{3}{4}).
    \end{align}
    Taylor expanding using the definition of the extension to $\singularBoundary$, as well as the definition of $x_0$, we have
    \begin{align}
        \begin{split} \label{E:TAYLOREXPANSIONBURGERSNEARB}
           \soluBurgersClassical(\bar{t},x) - \soluBurgersClassical(\bar{t},\bar{x}) & = \soluBurgersClassical(x_0) - \soluBurgersClassical(\bar{x}_0) \\
           & = \dataBurgers'(\bar{x}_0)(x_0-\bar{x}_0) +O((x_0-\bar{x}_0)^2).
        \end{split}
    \end{align} Inserting \eqref{E:XCOORDINATENEARB} into \eqref{E:TAYLOREXPANSIONBURGERSNEARB} concludes the proof of \eqref{pt:expansionNearSingularBoundary}.

    The proof of \eqref{PT:EXPANSIONNEARS} follows from similar arguments as \eqref{pt:expansionNearSingularBoundary}. Using the same conventions as \eqref{pt:expansionNearSingularBoundary}, we have that $(\bar{t},\bar{x}) = (1,2)$ and $\bar{x}_0 = 0$. Since $\dataBurgers''(0) = 0$, we expand $x - \overline{x}$ as in \eqref{eq:intermediateExpansionNearSingularBoundary} to one order higher and find 

    \begin{align} \label{E:EXPANDINGXNEARS}
        x - 2 = \frac{2}{3!}x_0^3 + O(x_0^4) = \frac{1}{3}x_0^3 + O(x_0^4),
    \end{align}
    where we used that $\dataBurgers'(0) = -1$ and $\dataBurgers'''(0) = 2$ in \eqref{E:EXPANDINGXNEARS}. Inserting $x_0 = (3(x-2))^{1/3} + O((x-2)^{4/9})$, which follows from \eqref{E:EXPANDINGXNEARS}, into the Taylor expansion $\soluBurgersClassical(1,x) - \soluBurgersClassical(1,2) = - x_0 + O(x_0^2)$ concludes the proof of \eqref{PT:EXPANSIONNEARS}.
\end{proof}

The following lemma now describes the weak solution $\soluBurgersWeak$. For $x \neq 0$, it is defined as the function taking the constant value $\dataBurgers(x)$ along the characteristics $\charcurveL_x(t)$ up until just before the point where they intersect the line $t = \tfrac{1}{2} x$. From \eqref{E:OUTGOINGCHARCURVEFORMULA}, it is straightforward to check that this intersection occurs at $-\frac{x}{\dataBurgers(x)}$.

\begin{remark}[Weak solutions to Burgers' equation] \label{R:WEAKSOLUTIONSTOBURGERS}
Before stating the lemma, we mention that the definition of ``weak'' solution to Burgers equation in divergence form $\partial_t \soluBurgers + \frac{1}{2} \partial_x(2+\soluBurgers)^2 = 0$ used below is identical to Def. \ref{D:WEAKSOLUTIONS} with $\soluBurgers$ in place of $\partial_t \solu - 2 \partial_x \solu$. For example, given an open set $\Omega\subset \R^{1+1}$, we say $\soluBurgersWeak \in L^1_{\textnormal{loc}}(\Omega)$ is a weak solution to Burgers' equation if \eqref{E:WEAKSOLUTIONIDENTITY} holds with $\soluBurgers$ in place of $\partial_t \solu - 2 \partial_x \solu$. Similar statements hold for solutions to the Cauchy problem \eqref{E:WEAKSOLUTIONTOIVPIDENTITY} and Oleinik entropy solutions as in \eqref{E:OLEINIKCONDITION}. We omit the repetitive definitions.
\end{remark}

\begin{lemma}[Weak solution of Burgers' equation]
\label{lem:BurgersSolutionWeak}
    Let
    \begin{align}
        \Gamma_{x_0}(t)  
        =   \begin{cases}
                \charcurveL_0(t) & x_0 = 0, \, t\in [0,1], \\
                \charcurveL_{x_0}(t) & x_0 \neq 0, \, t \in \left[0,-\frac{x_0}{\dataBurgers(x_0)}\right),
            \end{cases}
    \end{align}
    where $\charcurveL_{x_0}(t)$ is as in \eqref{E:OUTGOINGCHARCURVEFORMULA}. Define $\soluBurgersWeak$ along $\Gamma_{x_0}(t)$ by 
        \begin{align} \label{E:DEFOFSOLUBURGERSWEAK}
            \soluBurgersWeak \circ \Gamma_{x_0}(t) := \dataBurgers(x_0).
        \end{align}
    Then $\soluBurgersWeak$ is a weak solution to the Cauchy problem \eqref{E:FirstOrderBurgers} on $\R^{1+1}$. 
    
Moreover, this is the unique $L^\infty$ Oleinik entropy solution. Namely, the following condition holds for almost every $x\in\Reals$ and $z,t\in(0,\infty)$:
    \begin{align}
    \label{eq:dropForOleinik}
        \soluBurgersWeak(t,x+z)-\soluBurgersWeak(t,x)
        \leq{}& C\left(1+\frac{1}{t}\right)z .
    \end{align}
    
\end{lemma}

\begin{proof}
    A small number of calculations are required to show that $\soluBurgersWeak$ is smooth away from $\shockwave$. Each characteristic $\Gamma_{x_0}(t)$ terminates at $t=-x_0/\dataBurgers(x_0)$, where 
    $x=x_0+t(2+\dataBurgers(x_0))$
    $=x_0 +2t -2(x_0/\dataBurgers(x_0))\dataBurgers(x_0)$
    $=2t$. Hence, each $\Gamma_{x_0}$ terminates on $\shockwave$ as in \eqref{eq:def:discontinuousshock}. For $x_0$, the characteristic $\gamma_{x_0}$ from Lemma \ref{L:BURGERSSOLUTIONSMOOTHONMGHD} are line segments of the same line, but terminating at $t=-1/\dataBurgers'(x_0)$ $=1+x_0^2$. The difference between these termination times is $1+x_0^2-x_0/\arctan(x_0)$, which has the same sign as
    $\arctan(x_0)(1+x_0^2)-x_0$, which vanishes at $x_0$ and has derivative $1+2x_0\arctan(x_0)-1$ $=2x_0\arctan(x_0)\geq 0$. Thus, the line segments $\Gamma_{x_0}$ terminate on $\shockwave$ before reaching the singular boundary $\singularBoundary$. Similarly, for $x_0<0$, a convexity argument similar to that in Lemma \ref{L:BURGERSSOLUTIONSMOOTHONMGHD} can be used to show the line segments $\gamma_{x_0}$ can be extended to $t=-1/\dataBurgers'(x_0)$ without intersecting each other, and these line segments reach $\shockwave$ at $t=-x_0/\dataBurgers(x_0)>0$, before reaching $t=-1/\dataBurgers'(x_0)$. Since the characteristic segments $\Gamma_{x_0}$ never intersect each other, except at $\shockwave$, it follows that $\soluBurgersWeak$ is smooth away from $\shockwave$. 
    
    To show that $\soluBurgersWeak$ is a weak solution, it remains to show that it satisfies the Rankine-Hugoniot condition. Observe that the $x$ coordinate along $\gamma_{x_0}$ is $x_0+t(2+\dataBurgers(x_0)$, which, when it reaches $\shockwave$ at $t=-x_0/\dataBurgers(x_0)$ is 
    $x_0-2x_0/\dataBurgers(x_0)-x_0$
    $=-2x_0\dataBurgers(x_0)$ is even in $x_0$. Thus, the mean of the left and right limits of $2+\soluBurgersWeak$ is $((2+\soluBurgers(-x_0))+(2+\soluBurgers(x_0)))=2$, which is the slope of $\shockwave$. Thus, from the discussion in \cite[Section 3.4]{Evans}, $\soluBurgersWeak$ satisfies the Rankine-Hugoniot condition for Burgers' equation and is a weak solution in $\Reals^2$. 

    Since $\dataBurgers$ is negative for $x_0>0$ and positive for $x_0<0$, it follows that $\soluBurgersWeak$ decreases across $\shockwave$, and it is smooth (and also decreasing in $x$) elsewhere, so it satisfies the condtion \eqref{eq:dropForOleinik}. Thus, by \cite[Section 3.4, Theorem 3]{Evans}, it is the unique weak solution on $\Reals^2$ satisfying the Oleinik entropy condition.
\end{proof}

\section{Proof of the main theorems}
\label{s:pfMainThm}

In this section, the main Theorems \ref{thm:main:MGHD} and \ref{T:UNIQUEWEAKSOLUTIONS} are proved by constructing classical and weak solutions to the wave equation \eqref{E:QNLW:INTRO} whose $\uLunit$ derivative as in \eqref{E:DOUBLENULLFRAME} are the classical and weak solutions constructed in Lemmas \ref{L:BURGERSSOLUTIONSMOOTHONMGHD}-\ref{lem:BurgersSolutionWeak}. In other words, since Burgers' equation only has one characteristic speed, a second characteristic speed of $-2$ in the direction of $\uLunit$ is added to construct the desired wave equation to prove the main theorems \ref{thm:main:MGHD}-\ref{T:UNIQUEWEAKSOLUTIONS}. The classical and weak solutions for the wave equation are constructed from the classical and weak solutions of Burgers' equation in Lemmas \ref{L:BURGERSSOLUTIONSMOOTHONMGHD}-\ref{lem:BurgersSolutionWeak}.

\subsection{Lemma on integrals of Burgers' solutions}
\label{ss:integralsOfBurgers}
This section proves a lemma that will be used to construct a weak solution of the wave equation from the weak solution of Burgers' equation.

\begin{lemma}[Explicit formula for integrals of $\soluBurgersWeak$]
\label{lem:explicitFormulaForIntegralsOfSoluBurgers}
Let $\soluBurgersWeak$ be the weak entropy solution of Burgers' equation from Lemma \ref{lem:BurgersSolutionWeak} and define
\begin{align}
\soluWeak(t,x)={}& \frac12 \int_{y=x}^{x+2t} \soluBurgersWeak\left(t+\frac12(x-y),y\right) \di y .
\label{eq:soluWaveAsIntegral}
\end{align}

It follows that, for $t\leq\max\left(\frac{x}{2},2-\frac{x}{2}\right)$,
\begin{subequations}
\begin{align}
\partial_x\soluWeak(t,x)
={}&\ln\frac{4+\dataBurgers(x+2t)}{4+\soluBurgersWeak(t,x)}
-\frac12 \soluBurgersWeak(t,x) + \frac12 \soluBurgersWeak(0,t+2x)
\label{eq:partialxSoluWaveExplicit:below}
\end{align}
and for $t>\max\left(\frac{x}{2},2-\frac{x}{2}\right)$,
\begin{align}
\partial_x\soluWeak(t,x)
={}&\ln\frac{4+\dataBurgers(x+2t)}{4+\soluBurgersWeak(t,x)}
-\frac12 \soluBurgersWeak(t,x) + \frac12 \soluBurgersWeak(0,t+2x)\nonumber\\
& + \lim_{\epsilon\rightarrow0^+}\ln\frac{4+\soluBurgersWeak(x/4+t/2+\frac{\epsilon}{2},x/2+t-\epsilon)}{4+\soluBurgersWeak(x/4+t/2-\frac{\epsilon}{2},x/2+t+\epsilon)}
\label{eq:partialxSoluWaveExplicit:above}
\end{align}
\label{eq:partialxSoluWaveExplicit}
\end{subequations}
 For $(t,x)\not\in \shockwave$, 
\begin{align}
\uLunit \soluWeak(t,x) = \partial_t\soluWeak(t,x)
-2\partial_x\soluWeak(t,x) =\soluBurgersWeak(t,x).
\label{eq:partialtSoluWaveExplicit} 
\end{align}

In particular, for $x<2$ and $\epsilon>0$ small, 
\begin{align}
\partial_x\soluWeak\left(2-\frac{x}{2}+\epsilon,x\right)-\partial_x\soluWeak\left(2-\frac{x}{2},x\right)
\sim{}& (6\epsilon)^{1/2}.
\label{eq:partialxSoluWaveExplicit:roughness}
\end{align} 
\end{lemma}
\begin{proof}
Consider first the two regions considered in the statement of the lemma. Observe that $\shockwave$ is the subset of the line $t=x/2$ where $t > 1$, that $\CauchyHorizon$ is the subset of the line $t=2-x/2$ where $t> 1$, and that these two line segments intersect at $\initialSingularity$ where $t=1$ and $x=2$. Thus, the condition that $t\leq\max\left(\frac{x}{2},2-\frac{x}{2}\right)$ is the condition that the point $(t,x)$ lies beneath at least one of $\shockwave$ and $\CauchyHorizon$. 

For the readers convenience, recall \eqref{E:PARTIALXINTERMSOFDOUBLENULLFRAME}--\eqref{E:PARTIALTINTERMSOFDOUBLENULLFRAME}, which expresses the coordinate partial derivatives $\{\partial_t,\partial_x\}$ in terms of the $\gfour$-null frame $\{\Lunit,\uLunit\}$:
\begin{align}
\partial_x ={}& \frac{1}{4+\soluBurgers}(\Lunit -\uLunit), &
\partial_t ={}& \frac{2}{4+\soluBurgers}\Lunit
+\frac{2+\soluBurgers}{4+\soluBurgers}\uLunit .\label{eq:partialxPartialtInTermsOfNullVectors} 
\end{align}

Consider the value of $\soluWeak$ at a point $(t,x)=(t_0,x_0)$ in the case $t_0>\max(x_0/2,2-x_0/2)$. The other case follows similarly. In the case $t_0>\max(x_0/2,2-x_0/2)$, the characteristic curve of   $\uLunit = \partial_t - 2 \partial_x$ through the point $(t_0,x_0)$ is given by $t=(x_0/2+t_0)-x/2$. In particular, it crosses  $\shockwave$, given by $t=x/2$, when $x/2=t=(x_0/2+t_0)-x/2$, i.e. $x=x_0/2+t_0$. It crosses the $x$-axis when $0=t=(x_0/2+t_0)-x/2$, i.e. $x=x_0+2t_0$. The formula for $\solu$ can be rewritten as
\begin{align}
\soluWeak(t_0,x_0)
={}& \frac12 \int_{x_0}^{x_0/2+t_0} \soluBurgersWeak\left(t_0+\frac12(x_0-x),x\right) \di x \nonumber\\
 &+\frac12 \int_{x_0/2+t_0}^{x_0+2t_0} \soluBurgersWeak\left(t_0+\frac12(x_0-x),x\right) \di x .
\end{align}
For $t_0>\max(x_0/2,2-x_0/2)$, there is an open set around the right line segment $\{(t_0+x_0/2-x/2,x): x\in[\frac{x_0}{2} + t_0,x_0+2t_0]\}$ on which $\soluBurgersWeak$ can be extended smoothly (as $\soluBurgersClassical$). Because $\soluBurgersClassical$ is $C^2$, the difference quotients $(\soluBurgersClassical(t,x+h)-\soluBurgersClassical(t,x))/h$ are uniformly bounded on a neighbourhood and converge pointwise to $\partial_x\soluBurgersClassical(t,x)$, which coincides with $\partial_x\soluBurgersWeak(t,x)$, except at $t= \frac{x}{2}$, where the latter is undefined. Thus, the partial derivative can be passed through the integral. A similar argument applies on the left line segment $\{(t_0+x_0/2-x/2,x): x\in[x_0,\frac{x_0}{2}+t_0]\}$. 

Since $\Lunit \soluBurgersWeak = 0$ holds pointwise everywhere except for the line $t = \frac{x}{2}$, the following holds from the equation for $\partial_x$ in the equations \eqref{eq:partialxPartialtInTermsOfNullVectors}: 

\begin{align}
\partial_x\soluWeak(t_0,x_0)
={}& \frac12 \int_{x_0}^{x_0/2+t_0} \partial_x\soluBurgersWeak\left(t_0+\frac12(x_0-x),x\right) \di x
-\frac12 \soluBurgersWeak(t_0,x_0) \nonumber\\
& + \frac12 \int_{x_0/2+t_0}^{x_0+2t_0} \partial_x\soluBurgersWeak\left(t_0+\frac12(x_0-x),x\right) \di x
+\frac12 \soluBurgersWeak(0,t_0+2x_0)\\
={}&  \int_{x_0}^{x_0/2+t_0} \frac{1}{4+\soluBurgersWeak}\vecLbar\soluBurgersWeak\left(t_0+\frac12(x_0-x),x\right) \di x 
-\frac12 \soluBurgersWeak(t_0,x_0)\nonumber\\
 & + \int_{x_0/2+t_0}^{x_0+2t_0} \frac{1}{4+\soluBurgersWeak}\vecLbar\soluBurgersWeak\left(t_0+\frac12(x_0-x),x\right) \di x 
+\frac12 \soluBurgersWeak(0,t_0+2x_0)\\
 ={}& \frac12 \int_{x_0}^{x_0/2+t_0} \vecLbar\ln(4+\soluBurgersWeak)\left(t_0+\frac12(x_0-x),x\right) \di x 
-\frac12 \soluBurgersWeak(t_0,x_0)\nonumber\\
 & +\frac12\int_{x_0/2+t_0}^{x_0+2t_0} \vecLbar\ln(4+\soluBurgersWeak)\left(t_0+\frac12(x_0-x),x\right) \di x 
+\frac12 \soluBurgersWeak(0,t_0+2x_0)
\end{align}

Since $\vecLbar$ is tangent to the line of integration, the integrals can be evaluated by the fundamental theorem of calculus as
\begin{align}
\partial_x\soluWeak(t,x)
={}& \lim_{\epsilon\rightarrow0^+}\left(
\ln(4+\soluBurgersWeak)|_{x_0}^{x_0/2+t_0-\epsilon}
+\ln(4+\soluBurgersWeak)|_{x_0/2+t_0+\epsilon}^{x_0+2t_0}
\right) \nonumber\\
&-\frac12 \soluBurgersWeak(t_0,x_0)  +\frac12 \soluBurgersWeak(0,t_0+2x_0)\\
={}& \ln\frac{4+\soluBurgersWeak(0,x_0+2t_0)}{4+\soluBurgersWeak(t_0,x_0)}
+ \lim_{\epsilon\rightarrow0^+}\ln\frac{4+\soluBurgersWeak(x_0/4+t_0/2+\frac{\epsilon}{2},x_0/2+t_0-\epsilon)}{4+\soluBurgersWeak(x_0/4+t_0/2-\frac{\epsilon}{2},x_0/2+t_0+\epsilon)} \nonumber\\
&-\frac12 \soluBurgersWeak(t_0,x_0)  +\frac12 \soluBurgersWeak(0,t_0+2x_0).
\end{align}
For $t \leq\max(x_0/2,2-x_0/2)$, the continuity of $\soluBurgersClassical$ at $(x_0/4+t_0/2,x_0/2+t_0)$ means that the second term is absent. This completes the proof of formula \eqref{eq:partialxSoluWaveExplicit} except on the half line $\{(2-x/2,x): x<0\}$. 

Equation \eqref{eq:partialtSoluWaveExplicit} follows from a similar argument, using the equation for $\partial_t$ in the equations \eqref{eq:partialxPartialtInTermsOfNullVectors}.

It remains to prove equations \eqref{eq:partialxSoluWaveExplicit}-\eqref{eq:partialtSoluWaveExplicit} on the half line $\{(2-x/2,x):x<0\}$. Because the integrand in equation \eqref{eq:soluWaveAsIntegral} is uniformly bounded and pointwise convergent in $t$ (except at $x=2t$ for $t>1$), it follows that $\soluWeak$ is continuous. For $x<2t$ fixed, the partial derivative $\partial_t\soluWeak$ is continuous for $t\not=2-x/2$, the two one-sided limits of $\partial_t\soluWeak$ exist and are equal at $t=2-x/2$. Thus, by the fundamental theorem of calculus, the common value is the derivative at $t=2-x/2$. Similarly, the common value of the one-sided limits of $\partial_x\soluWeak$ give the value on the half line $\{(2-x/2,x):x<2\}$. 

For $t_0 > 2-\frac{x_0}{2}$, the discontinuity of $\soluBurgersWeak$ at $(x_0/4+t_0/2,x_0/2+t_0)$ means that the term involving the limit $\epsilon\rightarrow0^+$ in equations \eqref{eq:partialxSoluWaveExplicit:above} is nonzero. In the limit $t\rightarrow 2-x/2$, the term involving $\epsilon\rightarrow0^+$ vanishes as $(t-2+x/2)^{1/2}$ when $t-2+x/2$ is small by Lemma \ref{lem:singularBehaviourOfBurgersAtSingularBoundary} and the equality of $\soluBurgersClassical$ and $\soluBurgersWeak$ when $t<\max(x_0/2,2-x_0/2)$, which proves the remaining claim. 
\end{proof}

\subsection{Proof of Theorem \ref{thm:main:MGHD}}
\label{ss:proofOfThmMGHD}

This section proves Theorem \ref{thm:main:MGHD}. Corollary \ref{cor:causalBubbles} will be proved in Section \ref{SS:CAUSALBUBBLES}.

\begin{proof}[Proof of Theorem \ref{thm:main:MGHD}]
Recall $\soluBurgersClassical$ defined in Lemma \ref{L:BURGERSSOLUTIONSMOOTHONMGHD}. Define, for $(t,x)\in\soluDomainClassical$, 
\begin{align}
\soluClassical(t,x) :={}& \frac12 \int_{y=x}^{x+2t} \soluBurgersClassical\left(t+\frac12(x-y),y\right) \di y .
\label{eq:soluWaveAsIntegral:Classical}
\end{align}
Since $\soluBurgersClassical$ is smooth, so is $\soluClassical$. Differentiating with respect $\partial_t-2\partial_x$, one finds 
\begin{align}
\partial_t\soluClassical-2\partial_x\soluClassical={}&\soluBurgersClassical .
\end{align}
Since $\soluBurgersClassical$ is a solution of Burgers' equation, from the same argument as in equation \eqref{E:QNLWINDIVERGENCEFORM}, one finds $\soluClassical$ is a classical solution of the wave equation \eqref{E:QNLW:INTRO}. Furthermore, by definition $\soluClassical(0,x)=0$, and, since $\partial_x\soluClassical(0,x)=0$, we find $\partial_t\soluClassical(0,x)$ $=\soluBurgersClassical(0,x)$ $=-\arctan(x)$, so $\soluClassical$ satisfies the initial condition \eqref{E:QNLW:DATA:INTRO}. Since $\soluBurgersClassical$ is constant along characteristics, its values remain in $(-\pi/2,\pi/2)$. The acoustical metric remains Lorentzian as long $\soluBurgersClassical$ is not $2\not\in[-\pi/2,\pi/2]$. This proves the claim in point \eqref{pt:classicalProperies:existence}. 

We now prove the uniqueness property of the classical solution and its MGHD in point \eqref{pt:classicalProperties:MGHDuniqueness}. In Lemma \ref{lem:topologyGeometryOfSingularity}, it was already observed that $\boundary\soluDomainClassical=\CauchyHorizon\cup\initialSingularity\cup\singularBoundary$. Both $\singularBoundary$ and $\CauchyHorizon$ can be viewed as curves with domain $t\in(1,\infty)$. The characteristics have speed either $-2$ or $\soluBurgersClassical$, so any timelike inextendible curve must be parametrised by $t$ and continue until it reaches $\boundary\soluDomainClassical$; in particular, it must cross $t=0$. From our definition of a classical solution, it follows that $(\soluDomainClassical,\gfour)$ is globally hyperbolic with $\hsfcData$ as a Cauchy hypersurface.

Suppose there is a globally hyperbolic extension. This supposed extension would have to include a neighbourhood of a point either on $\CauchyHorizon$ or $\initialSingularity\cup\singularBoundary$. However, from the behaviour of $\soluBurgersClassical(1,x)$ $=\soluBurgersWeak(1,x)$ for small $|x-2|$ in equation \eqref{E:EXPANSIONNEARS}
and of $\soluBurgersClassical(t,x)$ for $t>1$ and $x>\xSingularBoundary(t)$ in equation \eqref{eq:expansionNearSingularBoundary}, we see that $\soluBurgersClassical$ fails to extend as a $C^1$ solution to any point on $\initialSingularity\cup\singularBoundary$, so $\soluClassical$ can't be  extended in $C^2$ to a neighbourhood of a point on $\initialSingularity\cup\singularBoundary$. If there were an extension to a neighbourhood of a point on $\CauchyHorizon$, then at any point on $\CauchyHorizon$, the metric given by $\gfour$ would be continuous, so $-2$ would remain a characteristic speed, and there would be a backward causal curve following $\CauchyHorizon$ back either to a point in the boundary of the domain of the extension of $\soluClassical$ or to $\initialSingularity=(1,0)$, either of which contradicts the global hyperbolicity of the domain of the extension, and, hence, contradicts our definition of classical solution. Thus, there can be no globally hyperbolic extension, and $(\soluDomainClassical,\soluClassical)$ is a maximal globally hyperbolic development. Since the boundary of $\soluDomainClassical$ is a piecewise $C^1$ curve (hence continuous), with $\soluDomainClassical$ lying on one side and its complement on the other, by \cite[Theorem 4.82]{EperonReallSbierski}, it follows that $(\soluDomainClassical,\soluClassical)$ is the unique maximal globally hyperbolic development. This proves point \eqref{pt:classicalProperties:MGHDuniqueness}. 

We now prove the statements concerning the boundary of the MGHD in point \eqref{pt:classicalProperties:characterisationOfTheBoundary} of Theorem \ref{thm:main:MGHD}. 

We've already seen that $\boundary\soluDomainClassical=\CauchyHorizon\cup\initialSingularity\cup\singularBoundary$, and that $\CauchyHorizon$, $\initialSingularity$, and $\singularBoundary$ are disjoint be definition, which proves \eqref{pt:classicalProperties:singularAndRegularBoundary}. In passing, we note that it was already shown that $\soluClassical$ fails to be $C^2$ at any point on $\initialSingularity\cup\singularBoundary$ and that, from Lemma \ref{lem:explicitFormulaForIntegralsOfSoluBurgers}, since $\soluClassical=\soluWeak$ on $\soluDomainAgree$, it follows that the partial derivatives of $\soluClassical$ extend smoothly to $\CauchyHorizon$ (and indeed, can be extended smoothly through $\CauchyHorizon$ by extending the formula given beneath $\CauchyHorizon$).

Equation \eqref{E:PARAMETRIZEDSINGULARBOUNDARY} is simply a reparametrization of the equation \eqref{eq:def:singularBoundary} which proves point \eqref{pt:classicalProperties:parametrizeSingularBoundary}. 

Both $\CauchyHorizon$ and $\singularBoundary$ are continuous curves with limit as $t\searrow0$ being $\initialSingularity=(1,0)$, which proves point \eqref{pt:classicalProperties:emanateFromInitialSingularity}. 
At $\CauchyHorizon$, from equations \eqref{eq:partialxSoluWaveExplicit:below} and \eqref{eq:partialtSoluWaveExplicit}, for $t<1-x/2$, $\soluWeak=\soluClassical$ have derivatives that extend smoothly to $\CauchyHorizon$ from below, which proves point \eqref{pt:classicalProperties:smoothAtCauchyHorizon}. 
That $\soluClassical$ extends as a  $C^{1,1/3}$ and $C^{1,1/2}$ function respectively to $\initialSingularity$ and $\singularBoundary$ follows
from Lemma \ref{lem:topologyGeometryOfSingularity} (particularly equations \eqref{eq:expansionNearSingularBoundary}--\eqref{E:EXPANSIONNEARS}), as well as $\Lunit \soluBurgersClassical = 0$ and computations for derivatives of $\soluClassical$ analogous to Lemma \ref{lem:explicitFormulaForIntegralsOfSoluBurgers}. These asymptotics in fact show that $\soluClassical$ cannot be better $C^{1,1/2}$ and $C^{1,1/3}$ at $\singularBoundary$ and $\initialSingularity$ respectively. In particular, the second derivatives must diverge sequentially, which concludes the point \eqref{pt:classicalProperties:regularityAtSingularBoundary}.

The definitions of the acoustical metric $\gfour$ as in \eqref{EqAcMetric} and the vectorfield $\vecL$ as in \eqref{E:DOUBLENULLFRAME} imply that they have the same regularity as $\soluBurgersClassical$ and hence also extend to $\initialSingularity\cup\singularBoundary$. We can then measure the length of the extended $\Lunit$ with respect to the extended $\gfour$ to be $\gfour(\Lunit,\Lunit) = 0$, which is an immediate consequence of \eqref{E:INVERSEMETRICINTERMSOFDOUBLENULLFRAME}. We have therefore proved point \eqref{pt:classicalProperties:regularityOfMetric}. 

Using the extended metric $\gfour$ from point \eqref{pt:classicalProperties:regularityOfMetric} to $\CauchyHorizon$, we can also measure the $\gfour$-length to of $\uLunit$, which is a smooth vectorfield on all of $\mathbb{R}^{1+1}$, to be $\gfour(\uLunit,\uLunit) = 0$. Since $\uLunit = \partial_t -2\partial_x$ is tangent to $\CauchyHorizon$ (see \eqref{eq:def:cauchyhorizon}), it follows that $\CauchyHorizon$ is $\gfour$-null. That the extended $\Lunit$ is tangent to $\singularBoundary$ follows immediately from  Lemma \ref{lem:topologyGeometryOfSingularity}, point \eqref{pt:tangentOfSingularBoundary}. This concludes the proof of point \eqref{pt:classicalProperties:boundaryIsNull}.

Finally, point \eqref{PT:NONUNIQUEINTEGRALCURVES} follows from the fact that, at any point $p \in \singularBoundary$, the extended $\Lunit$ is tangent to it. Hence, the backwards ODE initial value problem \eqref{E:ODEIVPFOREXTENDEDL} can be solved by going backwards along $\singularBoundary$ or backwards along the characteristic from within $\soluDomainClassical$ which was extended to terminate at $p \in \singularBoundary$.

\end{proof}

\subsection{Proof of Theorem \ref{T:UNIQUEWEAKSOLUTIONS}}
\label{ss:proofOfUniqueWeakSolution}

\begin{proof}[Proof of Theorem \ref{T:UNIQUEWEAKSOLUTIONS}]

Let $\soluBurgersWeak$ be the unique Oleinik entropy solution of the IVP \eqref{E:FirstOrderBurgers} to Burgers' equation constructed in Lemma \ref{lem:BurgersSolutionWeak}. Let $\soluWeak$ be defined by \eqref{eq:soluWaveAsIntegral}. Then, that $\soluWeak$ is a weak solution of the quasilinear wave equation \eqref{E:QNLWINDIVERGENCEFORM} is equivalent to $\soluBurgersWeak$ being a weak solution of Burgers' equation, which was proved in Lemma \ref{lem:BurgersSolutionWeak}, see also Remark \ref{R:WEAKSOLUTIONSTOBURGERS}. Furthermore, the uniqueness in the class of Oleinik entropy solutions in Definition \ref{D:CLASSOFENTROPYSOLUTIONS} is equivalent to $\soluBurgers$ being a weak solution in the class considered in Lemma \ref{lem:BurgersSolutionWeak}. This proves point \eqref{pt:weakSolution:existenceUniqueness}.

The classical Burgers' solution $\soluBurgersClassical$ and weak Burgers' solution $\soluBurgersWeak$ differ in the region $\soluDomainClassical\backslash\soluDomainAgree$ $=\BurgersDomainRight\backslash\Omega_3$ bounded from below by the shock $\shockwave$ and above by the singular boundary $\singularBoundary$. In fact, $\soluBurgersClassical<\soluBurgersWeak$ everywhere in this region. Since $\soluClassical$ and $\soluWeak$ are defined by integrating along the lines of $\vecLbar$, and the $\vecLbar$ directional derivative of $\soluClassical$ is a.e. less than the $\vecLbar$ derivative of $\soluWeak$, it follows that $\soluClassical$ and $\soluWeak$ differ at every point in this region, which proves point \eqref{pt:weakDoesntExtendClassical}. 

Since $\soluBurgersClassical$ and $\soluBurgersWeak$ agree on $\soluDomainAgree$ (the region beneath $\CauchyHorizon$ and $\shockwave$), it follows that $\soluClassical$ and $\soluWeak$ agree in this domain. The global hyperbolicity of $\soluDomainAgree$ follows from a similar argument to the global hyperbolicity of $\soluDomainClassical$ in the proof of Theorem \ref{thm:main:MGHD}, that is that all causal curves in the region must have bounded slope and hence intersect $t=0$, making $t=0$ a Cauchy hypersurface, and hence $\soluDomainAgree$ globally hyperbolic (although not maximally globally hyperbolic). This proves point \eqref{pt:weakSolution:characterisationOfDomain}. 

For the proof of point \eqref{pt:weakSolution:shockDevelopment} of this theorem, consider the conditions in Definition \ref{D:SHOCKDEVELOPMENT} for the shock development problem. Set $\soluDomainSDP := \big(\R^{1+1}\big) \setminus \soluDomainAgree$ and $(\soluBurgersSDP,\soluSDP) := \big(\soluBurgersWeak|_{\soluDomainSDP},\soluWeak|_{\soluDomainSDP}\big)$. Consider the curves $x=2t$ and $x=\xSingularBoundary(t)$ $= \xSingularBoundaryFromData(\sqrt{t-1})$ for $t\geq 1$, with $\xSingularBoundaryFromData$ from equation \eqref{E:XCOORDINATEOFSINGULARSETFROMDATA}. Both curves start at $t=1$ and $x=2$. By direct computation, $\di \xSingularBoundary/\di t$ $=2-\arctan\sqrt{t-1}$. Since $\di \xSingularBoundary/\di t(1)=1$ and, for $t>1$, $\di \xSingularBoundary/\di t(t)<1$, it follows that the line segment $\shockwave$ given by $x=2t$ initially coincides with and is tangent to the singular boundary $\singularBoundary$ given by the curve $x=\xSingularBoundaryFromData(\sqrt{t})$, but that the singular boundary $x=\xSingularBoundaryFromData(\sqrt{t})$ lies to the left, and hence above, the shockwave $\shockwave$ for $t>1$.  This proves points \eqref{pt:SDP:shockwaveInPastOfSingularBoundary}, \eqref{pt:SDP:emanateFromInitialSingularity}, and \eqref{pt:SDP:asymptoticallyNull} from the definition of the shock development problem. Consider now the two characteristics for the Burgers' equation launched from $(t,x)=(0,\pm x_0)$. These are given respectively by $x=(2\mp\arctan(x_0))t\pm x_0$ and intersect at $t=x_0/\arctan(x_0)$, where $x=2x_0/\arctan(x_0)$ $=2t$. Thus, at each point on $\shockwave$, the characteristics of $\soluBurgersWeak$ arriving at a given $t$ arise from $x_0$ and $-x_0$ for some value of $x_0$. In particular, the mean of the slopes is $\frac12((2+\arctan(x_0))+(2-\arctan(x_0)))$ $=2$, which,  given the parametrization $\shockwave = \{ (t,k(t) = 2t) \, : \, t > 1\}$ of the shock, means  the Rankine-Hugoniot condition \eqref{E:RANKINEHUGONIOT} hold. Since the weak and classical solutions agree on the Cauchy horizon $\CauchyHorizon$, we have now proved point \eqref{pt:SDP:RankineHugoniot} in the definition of the shock development problem. To prove the geometric determinism condition, from the explicit parametrization of $\shockwave$, it follows that $T:= \partial_t + 2 \partial_x$ spans its tangent space at any point. Hence, using the explicit form of the metric \eqref{EqAcMetric}, we compute $\gfour(T,T) = \gfour_{\alpha\beta} T^\alpha T^\beta = - \frac{16 \soluBurgers}{(4+\soluBurgers)^2}$. Now, within the MGHD $\soluDomainClassical$, the characteristics $(t,x_0 + t(2+\dataBurgers(x_0)))$ intersecting $\shockwave$ are those emanating from $x_0 > 0$. Since $\soluBurgersClassical$ is constant on these characteristics, and $\dataBurgers(x_0) = - \arctan(x_0) < 0$ for $x_0>0$, we have that $\gfour(\soluBurgersClassical)_{\alpha\beta} T^\alpha T^\beta > 0$ everywhere on $\shockwave$ and hence it is $\gfour(\soluBurgersClassical)$-spacelike. Contrastingly, since $\dataBurgers(x_0) = - \arctan(x_0) >0$ for $x_0<0$, $\soluBurgersSDP=\soluBurgersWeak|_{\soluDomainSDP}$, and $\soluBurgersWeak$ is defined on $\soluDomainSDP$ as the constant values from characteristics emanating from $x_0 < 0$ (see \eqref{E:DEFOFSOLUBURGERSWEAK}), by continuously extending these to $\shockwave$ from the left, it follows that $\gfour(\soluBurgersSDP)_{\alpha\beta} T^\alpha T^\beta < 0$. This proves that $\shockwave$ is $\gfour(\soluBurgersSDP)$-timelike, which concludes the proof of the geometric determinism condition part \eqref{pt:SDP:supersonic}  of the shock development problem in Definition \ref{D:SHOCKDEVELOPMENT} and point \eqref{pt:weakSolution:shockDevelopment} of Theorem \ref{T:UNIQUEWEAKSOLUTIONS}.

Trivially from the discussion of the geometric determinism condition, based on the signs of  $\soluBurgersClassical$ and $\soluBurgersSDP$ restricted to $\shockwave$ (from the left in the latter), the geometric determinism condition is equivalent to the slope $k'$ of the shockwave $\shockwave$ lying between the slopes of the left and right characteristics reaching the shockwave, which is exactly the Lax entropy condition \eqref{E:LAXENTROPY}; this proves point \eqref{pt:weakSolution:geometricDeterminism} of Theorem \ref{T:UNIQUEWEAKSOLUTIONS}.

Finally, from from equations \eqref{eq:partialxSoluWaveExplicit}-\eqref{eq:partialtSoluWaveExplicit}, we find $\soluWeak$ is $C^1$ at the Cauchy Horizon $\CauchyHorizon$, but, from equation \eqref{eq:partialxSoluWaveExplicit:roughness}, $\soluWeak$ cannot be better than $C^{1,1/2}$, which completes the proof of point \eqref{pt:weakSolution:weakSingularityAcrossCauchyHorizon} and hence the theorem.
\end{proof}

\subsection{Lorentzian causality in the MGHD and causal bubbles} \label{SS:CAUSALBUBBLES}

This section studies the Lorentzian causality of the closure $\textnormal{cl}(\Omega_C) = \Omega_C \cup \mathcal{B} \cup \mathcal{S} \cup \underline{\mathcal{C}}$ of the MGHD. Recall that by Theorem \ref{thm:main:MGHD} the acoustical metric $\gfour$ is smooth in $\Omega_C$ and extends continuously to $\textnormal{cl}(\Omega_C)$ -- by slight abuse of notation we denote this continuous extension again by $\gfour$. We then consider $(\textnormal{cl}(\Omega_C), \gfour)$, which may be viewed as a Lorentzian manifold with piecewise $C^1$ boundary and continuous metric.

First, we lay out the definitions of causal and timelike past of a point. Recall from Appendix \ref{SecApp} that the set of timelike vectors at any point of $\textnormal{cl}(\Omega_C)$ forms a disconnected double cone. We call the component whose closure contains $L$ and $\underline{L}$ the future and the other one the past. Non-zero causal vectors lying in the closure of the future component are called \textbf{future directed causal}. The terms $\textbf{past directed causal/timelike}$ etc.\ are defined analogously. A piecewise $C^1$-regular curve $\gamma : I \to \textnormal{cl}(\Omega_C)$ is called past directed timelike/causal if $\dot{\gamma}(s)$ is past directed timelike/causal for all $s \in I$.
We can now define for $p \in \textnormal{cl}(\Omega_C)$ the \textbf{timelike (causal) past of $p$}, denoted by $I^-(p, \textnormal{cl}(\Omega_C))$ ($J^-(p, \textnormal{cl}(\Omega_C)$), to be the set consisting of all points $q \in \textnormal{cl}(\Omega_C)$ such that there exists a past directed timelike (causal) piecewise $C^1$-curve in $\textnormal{cl}(\Omega_C)$ from $p$ to $q$. The timelike (causal) future may be defined analogously, but is not needed here.

An interesting phenomenon present in the closure $(\textnormal{cl}(\Omega_C), \gfour)$ of our MGHD is the \emph{dynamical formation} of a `causal bubble'. This concept was introduced in \cite{ChruscielGrant} and may be defined here as the existence of a point $p \in \textnormal{cl}(\Omega_C)$ for which $J^-(p, \textnormal{cl}(\Omega_C)) \setminus I^-(p, \textnormal{cl}(\Omega_C))$ has non-empty interior. Such behaviour can only happen for Lorentzian metrics below Lipschitz regularity (see \cite{ChruscielGrant}) and indeed does happen in our example at the singular boundary $\mathcal{B}$, see Figure\,\ref{F:CAUSALBUBBLE}:
for any point $p \in \mathcal{B}$ we may consider the integral curves of $-L = -\partial_t - (2+\soluBurgersClassical) \partial_x $ through $p$. They correspond to past directed null curves. Note, however, that there is not a unique such curve. For example, one may enter $\Omega_C$ directly and follow the backwards characteristic arising from Burgers' equation with slope $-(2+\soluBurgersClassical)$ in $\Omega_C$. On the other hand one may also first follow the singular boundary $\singularBoundary$ up to $\mathcal{S}$ and only then enter $\Omega_C$ to follow the backwards characteristic arising from Burgers' equation with slope $-(2+\soluBurgersClassical)$ in $\Omega_C$. The causal past $J^-(p, \textnormal{cl}(\Omega_C))$ is bounded to the left by this latter curve and to the right by the (unique) integral curve of $-\underline{L}$ to the past of $p$.


Points in the timelike past $I^-(p, \textnormal{cl}(\Omega_C))$ on the other hand have to be connected to $p$ by a timelike curve. This rules out that this curve may trace out part of the singular boundary $\mathcal{B}$. Thus, $I^-(p, \textnormal{cl}(\Omega_C))$ is bounded to the left by the backwards characteristic arising from Burgers' equation with slope $-(2+\soluBurgersClassical)$ in $\Omega_C$; the first of the integral curves of $-L$ discussed above. To the right, it is bounded by the same integral curve of $-\underline{L}$ to the past. This shows that a causal bubble is present at every $p \in \mathcal{B}$ and proves Corollary \ref{cor:causalBubbles}. We also refer the reader to \cite{Christodoulou:shockDevelopment,AbbresciaSpeck} for an analogous phenomenon in the study of the Euler equations.

Finally, we discuss another interesting property of the singular boundary which amounts to an alternative perspective on causal bubbles. Clearly, with respect to the intrinsic geometry, $\mathcal{B}$ is null. By `intrinsic geometry' we just mean the restriction of $\gfour$ to $\mathcal{B}$. On the other hand, if we only consider the exterior geometry, i.e., $(\Omega_C, \gfour)$, then the geometry in the vicinity of $\mathcal{B}$ would suggest, by analogy with smooth Lorentzian geometry, that $\mathcal{B}$ is spacelike. Let us explain this point in more detail: consider a smooth Lorentzian manifold and pick a smooth null hypersurface and smooth spacelike one. For sake of concreteness say we consider the hypersurfaces $\{t=x\}$ and $\{t=0\}$ in $1+1$-dimensional Minkowski spacetime, see Figures \ref{FigNull} and \ref{FigSpacelike}.
\begin{figure}[h]
\centering
\begin{minipage}{.5\textwidth}
  \centering
 \def\svgwidth{2.8cm}
\begingroup%
  \makeatletter%
  \providecommand\color[2][]{%
    \errmessage{(Inkscape) Color is used for the text in Inkscape, but the package 'color.sty' is not loaded}%
    \renewcommand\color[2][]{}%
  }%
  \providecommand\transparent[1]{%
    \errmessage{(Inkscape) Transparency is used (non-zero) for the text in Inkscape, but the package 'transparent.sty' is not loaded}%
    \renewcommand\transparent[1]{}%
  }%
  \providecommand\rotatebox[2]{#2}%
  \newcommand*\fsize{\dimexpr\f@size pt\relax}%
  \newcommand*\lineheight[1]{\fontsize{\fsize}{#1\fsize}\selectfont}%
  \ifx\svgwidth\undefined%
    \setlength{\unitlength}{131.63146853bp}%
    \ifx\svgscale\undefined%
      \relax%
    \else%
      \setlength{\unitlength}{\unitlength * \real{\svgscale}}%
    \fi%
  \else%
    \setlength{\unitlength}{\svgwidth}%
  \fi%
  \global\let\svgwidth\undefined%
  \global\let\svgscale\undefined%
  \makeatother%
  \begin{picture}(1,0.99999992)%
    \lineheight{1}%
    \setlength\tabcolsep{0pt}%
    \put(0,0){\includegraphics[width=\unitlength,page=1]{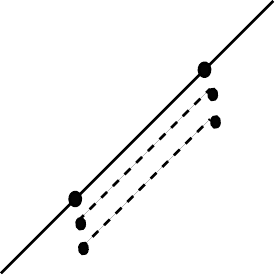}}%
    \put(0.18407108,0.35457221){\color[rgb]{0.01568627,0.01568627,0.01568627}\makebox(0,0)[lt]{\lineheight{1.25}\smash{\begin{tabular}[t]{l}$q$\end{tabular}}}}%
    \put(0.65545695,0.82595832){\color[rgb]{0.01568627,0.01568627,0.01568627}\makebox(0,0)[lt]{\lineheight{1.25}\smash{\begin{tabular}[t]{l}$p$\end{tabular}}}}%
  \end{picture}%
\endgroup%

      \caption{The hypersurface $\{t = x\}$ in $\Reals^{1+1}$} \label{FigNull}
\end{minipage}%
\begin{minipage}{.5\textwidth}
  \centering
  \def\svgwidth{5cm}
\begingroup%
  \makeatletter%
  \providecommand\color[2][]{%
    \errmessage{(Inkscape) Color is used for the text in Inkscape, but the package 'color.sty' is not loaded}%
    \renewcommand\color[2][]{}%
  }%
  \providecommand\transparent[1]{%
    \errmessage{(Inkscape) Transparency is used (non-zero) for the text in Inkscape, but the package 'transparent.sty' is not loaded}%
    \renewcommand\transparent[1]{}%
  }%
  \providecommand\rotatebox[2]{#2}%
  \newcommand*\fsize{\dimexpr\f@size pt\relax}%
  \newcommand*\lineheight[1]{\fontsize{\fsize}{#1\fsize}\selectfont}%
  \ifx\svgwidth\undefined%
    \setlength{\unitlength}{261.42858791bp}%
    \ifx\svgscale\undefined%
      \relax%
    \else%
      \setlength{\unitlength}{\unitlength * \real{\svgscale}}%
    \fi%
  \else%
    \setlength{\unitlength}{\svgwidth}%
  \fi%
  \global\let\svgwidth\undefined%
  \global\let\svgscale\undefined%
  \makeatother%
  \begin{picture}(1,0.59011487)%
    \lineheight{1}%
    \setlength\tabcolsep{0pt}%
    \put(0,0){\includegraphics[width=\unitlength,page=1]{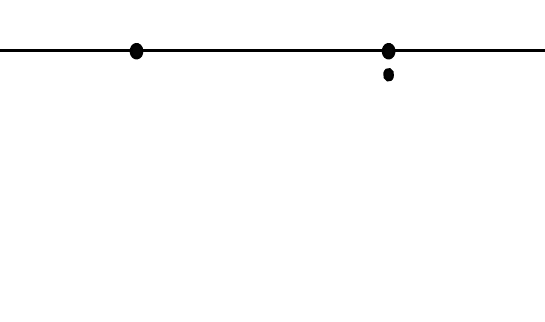}}%
    \put(0.24201742,0.54366462){\color[rgb]{0.01568627,0.01568627,0.01568627}\makebox(0,0)[lt]{\lineheight{1.25}\smash{\begin{tabular}[t]{l}$q$\end{tabular}}}}%
    \put(0.70922241,0.53933878){\color[rgb]{0.01568627,0.01568627,0.01568627}\makebox(0,0)[lt]{\lineheight{1.25}\smash{\begin{tabular}[t]{l}$p$\end{tabular}}}}%
    \put(0,0){\includegraphics[width=\unitlength,page=2]{SpacelikeBdry.pdf}}%
  \end{picture}%
\endgroup%

      \caption{The hypersurface $\{t = 0\}$ in $\Reals^{1+1}$} \label{FigSpacelike}
\end{minipage}
\end{figure}
First consider $\{t=x\}$ with two points $p$ and $q$ such that $p$ lies to the causal future of $q$. If we now only consider the exterior geometry to the past of $\{t = x\}$, then it has the property that if we shoot back a  left-going null geodesic from any point \emph{near} $p$ that it reaches points \emph{near} $q$, cf.\ Figure \ref{FigNull}. In the limit, these null geodesics exactly converge to $\{t = x\}$. This is in stark contrast to the situation for the spacelike hypersurface $\{t=0\}$: considering two points $p$ and $q$ on $\{t=0\}$ and a sequence of points $p_n$ converging to $p$ from the past, we can again shoot back left- (and right-) going null geodesics from $p_n$. However, this time these null geodesics stay uniformly bounded away from $q$. 

Now, going back to our exterior geometry $(\Omega_C, \gfour)$ and considering a sequence of points $p_n \in \Omega_C$ which converges to a point $p \in \mathcal{B}$, we see that the past directed null geodesics emanating from $p_n$ do not stay close to $\mathcal{B}$ as in our model case of a null hypersurface, Figure \ref{FigNull}, but veer off in the same manner as in Figure \ref{FigSpacelike}. It is in this way that the exterior geometry of  $\mathcal{B}$ suggests that $\mathcal{B}$ is spacelike.

\begin{center}
	\begin{figure}[h]  
      	\begin{overpic}[scale=1.2, grid = false, tics=5, trim=-.5cm 0cm -1cm -.5cm, clip]{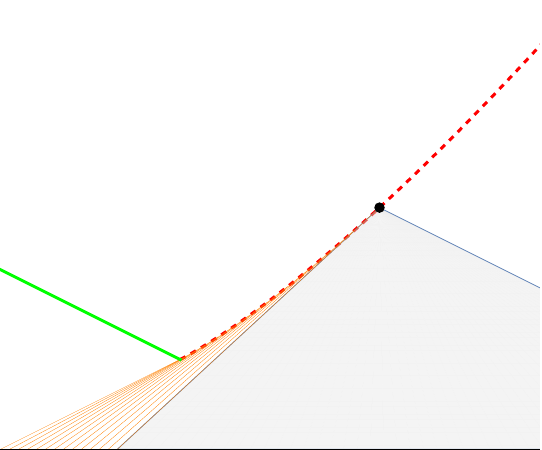}
       \end{overpic}
       \caption{The singular boundary is the dashed line, the Cauchy horizon the green one. The timelike past $I^-(p, \textnormal{cl}(\Omega_C))$ of an arbitrary point on the singular boundary is given by the shaded region in grey. The causal bubble $J^-(p, \textnormal{cl}(\Omega_C)) \setminus I^-(p, \textnormal{cl}(\Omega_C))$ is given by the region shaded by the characteristics corresponding to $L$ in orange.}
       \label{F:CAUSALBUBBLE}
    \end{figure}
\end{center}

\appendix

\section{Lorentzian geometry and globally hyperbolic developments} \label{SecApp}

This appendix presents the basic notions of Lorentzian geometry and the concept of a globally hyperbolic development in the context of the problem studied in this paper. The reader interested in how the geometric concepts generalise to abstract Lorentzian manifolds in full generality is referred to \cite{ONeill}. For the concept of a globally hyperbolic development of a general scalar quasilinear wave equation we refer the reader to \cite{EperonReallSbierski}.

We consider $\mathbb{R}^2$  and let $\Omega \subseteq \mathbb{R}^2$ be open. A \textbf{Lorentzian metric} $g$ on $
\Omega$ is a symmetric $2$-covariant tensor field on $\Omega$ with signature $(-1,+1)$. Paraphrasing this definition, for each $x \in \Omega$, $g_x \in \mathbb{R}^{2 \times 2}$ is a symmetric bilinear form on $\mathbb{R}^2$ with one negative and one positive eigenvalue. If the map $\Omega \ni x \mapsto g_x \in \mathbb{R}^{2 \times 2}$ is continuous (smooth, etc.\ ), the Lorentzian metric $g$ is said to be continuous (smooth, etc.\ ).  The pair $(\Omega, g)$ is an example of a Lorentzian manifold.

As an example, $\gfour$ as defined in \eqref{EqAcMetric} is a symmetric $2$-covariant tensor field. Its determinant can be computed to be $\det \gfour = -\frac{4}{(2+ \Psi)^2}$, which, for $\Psi \neq -2$ is negative. Thus, $\gfour$ is non-degenerate with one positive and one negative eigenvalue -- hence, it is a Lorentzian metric. Its regularity depends on the regularity of $\Psi$. If $\Psi$ is smooth, the acoustical metric is also smooth.

Given a vector $\ X \in T_x\Omega \simeq \mathbb{R}^2$ at a point $x \in \Omega$, we say that
\begin{equation*}
    X \textnormal{ is } \begin{cases}
        \textnormal{\textbf{timelike}} &\iff g_x(X,X) <0 \\
        \textnormal{\textbf{spacelike}} &\iff g_x(X,X) >0 \\
        \textnormal{\textbf{null} (or \textbf{characteristic}) } &\iff g_x(X,X) =0 \;.\\
    \end{cases}
\end{equation*}
Moreover, a vector is said to be \textbf{causal} if it is timelike or null. The set of timelike vectors at a given point $x \in \Omega$ forms a disconnected double cone. 

Consider a $C^1$-regular curve $\gamma : I \to \Omega$, where the interval $I \subseteq \Reals$ is of the form $(a,b)$, $[a,b)$, $(a,b]$, or $[a,b]$ with $-\infty \leq a < b \leq \infty$. We say that $\gamma$ is \textbf{extendible} (in $\Omega$) if at least one of the limits $\lim_{s \to a} \gamma(s)$ or $\lim_{s \to b}\gamma(s)$ exists \underline{in $\Omega$}. Let us emphasise that the notion of extendibility is with respect to the domain $\Omega$! If  a curve is not extendible in $\Omega$, then we say it is \textbf{inextendible} (in $\Omega$). Moreover, a $C^1$-regular curve $\gamma : I \to \Omega$ is called \textbf{timelike/null/spacelike/causal} if the vector $\dot{\gamma}(s)$ is timelike/null/spacelike/causal for all $s \in I$ with respect to the Lorentzian metric $g$ on $\Omega$.

Next, consider an open set $\Omega \subseteq \Reals^2$ which contains our initial data hypersurface $\Sigma = \{t=0\}$, together with a smooth Lorentzian metric $g$ on $\Omega$. We say that $\Sigma$ is a \textbf{Cauchy hypersurface} for $(\Omega, g)$, if every inextendible timelike $C^1$-curve in $\Omega$ intersects $\Sigma$ exactly once.  In this case, one also uses the terminology that $(\Omega, g)$ is \textbf{globally hyperbolic} with Cauchy hypersurface $\Sigma$.

So far, our definitions were purely geometric. We now make the connection with the Cauchy problem for the quasilinear wave equation \eqref{E:QNLW:INTRO}, \eqref{E:QNLW:DATA:INTRO}. A classical (smooth) solution $\Phi : \Sigma \subseteq \Omega \to \Reals$ of \eqref{E:QNLW:INTRO} which attains the initial data \eqref{E:QNLW:DATA:INTRO} is called a \textbf{development} of the initial data. Now, a development is called a \textbf{globally hyperbolic development} (GHD), if $\Omega$ together with the acoustical metric $g$, defined by \eqref{EqAcMetric}, is globally hyperbolic with Cauchy hypersurface $\Sigma$. And finally, a \textbf{maximal globally hyperbolic development} (MGHD) is a GHD $(\Omega, \Phi)$ such that there is no other GHD $(\Omega', \Phi')$ with $\Omega \subsetneq \Omega'$ and $\Phi'|_\Omega = \Phi$. In other words, there is no bigger GHD which extends $(\Omega, \Phi)$.

\section{Open problems and perspective}
\label{S:PROBLEMSANDPERSPECTIVE}
In this paper, we have studied various stable phenomena tied to shock formation, starting from smooth initial conditions, for the model quasilinear wave equation \eqref{E:QNLW:INTRO} in $1$ space dimension. There is a wealth of fascinating open problems connected to our work, and here we highlight some of the most compelling ones.

\begin{itemize}
    \item (\textbf{Other solution regimes for the model problem}) With very minor additional effort, the results we have presented for the model equation's Cauchy problem \eqref{E:QNLW:INTRO}--\eqref{E:QNLW:DATA:INTRO} hold for open sets of initial data which are small perturbations of \eqref{E:QNLW:DATA:INTRO}. But, what happens for other smooth initial data? In particular, are other kinds of singularities other than shocks possible? Are there any nontrivial global solutions? 
    \item (\textbf{Nonuniqueness of MGHDs for the model problem}) Theorem \ref{thm:main:MGHD} proves that the MGHD of the data \eqref{E:QNLW:DATA:INTRO} is unique, thanks to the boundary satisfying the global geometric one-sided condition of \cite{EperonReallSbierski} (see Theorem \ref{thm:main:MGHD} point \ref{pt:classicalProperties:MGHDuniqueness} and its proof in Sect. \ref{ss:proofOfThmMGHD}). Does there exist any smooth initial data that leads to non-unique MGHDs? Such MGHDs must have a boundary that is "bad" in some sense. If they do exist, are these unstable in the sense that, if you perturb the data, do you recover uniqueness?
    
    \item (\textbf{Structure of the boundary of MGHDs for the model problem})
    More generally, what can one say about the boundary of an MGHD arising from smooth generic initial data for the model problem? What kind of regularity can it have? Does it even have to be rectifiable? One might keep in mind the results of \cite{Kommemi}, which shows that for the Einstein--Maxwell--Klein--Gordon system in spherical symmetry, which is a quasilinear $1+1$-dimensional hyperbolic system, the boundary of the MGHD can in principle have many different kinds of components. 
    \item  (\textbf{Sensitivity to nonlinearities}) How sensitive are the results to perturbing the structure of the nonlinearities? 
        One should keep in mind that for Riccati's ODE $\dot{y} = y^2$, perturbing it by a cubic term to obtain
        $\dot{y} = y^2 + \epsilon y^3$ can dramatically alter the behavior of solutions, by either enhancing the blowup-rate or preventing the singularity altogether, depending on the sign of $\epsilon$. Similarly, one might also keep in mind the
        inhomogeneous Burgers equation $\partial_t \Psi + \Psi \partial_x \Psi = \Psi^2$ in one spatial dimension. For this equation, one can construct open sets of 
        smooth initial data in which $\partial_x \Psi|_{t=0}$ is large but $\Psi|_{t=0}$ itself is small in $L^{\infty}$, such that a shock forms in finite time, much 
        like in the homogeneous case. For the same equation, one can also construct different sets of smooth initial data in which 
        $\Psi|_{t=0}$ itself is large but $\partial_x \Psi|_{t=0}$ is small in $L^{\infty}$, such that the solution $\Psi$ itself blows up in finite time, much like what happens to solutions to Riccati's ODE $\dot{y} = y^2$ for positive times when $y(0) > 0$.
    \item (\textbf{Global uniqueness of weak solutions}) The global weak solution we constructed in Theorem \ref{T:UNIQUEWEAKSOLUTIONS} is unique within the class of Oleinik entropy solutions. Are there other criteria one can impose to guarantee uniqueness (ideally criteria that might be relevant for related higher dimensional problems)?
    \item  (\textbf{The compressible Euler equations in $\mathbb{R}^{1+1}$}) 
        Do similar MGHD/global weak solution results hold for the isentropic compressible Euler equations in $(1+1)$-dimensions? We conjecture that they do, but the proofs will not be identical. Specifically, for our model problem written as a system of two transport equations in transverse characteristic directions \eqref{E:FirstOrderBurgers}--\eqref{E:FirstOrderWavePartshockdevelopmentp}, the characteristic direction $\vecLbar = \partial_t - 2\partial_x$ is solution-independent. In $(1+1)$-dimensions, the isentropic compressible Euler equations can also be written as a similar system $\Lunit_+ \mathcal{R}_+ = 0$ $\uLunit_- \mathcal{R}_- = 0$,  where $\mathcal{R}_\pm$ are the famous Riemann invariants and $\Lunit_\pm$ are the characteristic directions and are \emph{both} solution dependent. This question is significantly harder in $(1+3)$-dimensions because there is a balancing act between dispersion (which makes characteristics want to spread out) and shock formation (which makes characteristics want to collapse with infinite density).

    \item (\textbf{Global weak solutions in multi-dimensions}) The only multi-dimensional result for
        global, weak, shock-containing solutions to 
        a quasilinear hyperbolic PDE without symmetry assumptions
        is the recent paper \cite{ginsberg2024stability} by Ginsberg--Rodnianski, which concerns quasilinear wave equations in $3$ space dimensions. The initial data treated in \cite{ginsberg2024stability} are (generally asymmetric) perturbations of spherically symmetric initial data of the isentropic Euler equations containing two shocks, and the resulting weak solutions are piecewise smooth in the regions between two shock hypersurfaces. 
        This result can be viewed as an global version of Majda \cite{Majda}/ Majda-Tohmann's \cite{MajdaThomann} celebrated local existence results for weak solutions arising from piecewise smooth data. However, we note that \cite{ginsberg2024stability} does not treat the full Euler system as it does not account for the physical jump in entropy and vorticity across the shocks. It is of prime interest to extend this result to the full Euler system, other hyperbolic equations, and to understand the global structure of weak solutions.

    \item (\textbf{Structure of the shock in multi-dimensions}) 
     The shock curve in the model problem from Theorem \ref{T:UNIQUEWEAKSOLUTIONS} is smooth, and the weak solution has smooth one-sided limits, see Fig. \ref{fig:BurgersBlowUp}. This is also the case for the multi-dimensional shock hypersurfaces constructed in \cite{Majda, MajdaThomann,ginsberg2024stability}.  However, in multi-$D$, this is not always the case. It has been known since the seminal work of mathematical physicist Guderley \cite{guderley1942starke} that solutions to the compressible Euler equations with a radial shock could collapse into an implosion singularity, where the undifferentiated solution variables blow-up. These have now been rigorously proved to exist in the recent work of \cite{jang2025self}. 
     
     On the other hand, upcoming work of Anderson-Angelopoulos-Chaturvedi proves that a weak solution to the $2D$-Bugers equation could develop infinite gradients, along one side of an \emph{already existing} shock hypersurface while being one-sided smooth on the other side of the shock. 

     Can one fully classify and rigorously construct the possible structures, degeneracies, and singularities that can develop on the shock?

    \item (\textbf{Global structure of MGHD in multi-dimensions}) To date, the only multi-dimensional results with explicit constructions of the MGHD of shock-forming data is \cite{AbbresciaSpeck,AbbresciaSpeck2,shkoller2024geometry}. However, as we stated in Sect. \ref{SSS:HISTORICALCONTEXTANDSUBTLETIESOFMGHDS}, these results only describe a \emph{local} pre-compact portion of the MGHD. In particular, the global geometric condition which is sufficient for uniqueness from \cite{EperonReallSbierski} cannot be deduced by the results of \cite{AbbresciaSpeck,AbbresciaSpeck2,shkoller2024geometry}. In particular, one could imagine that shock-forming data develops the (part) of an MGHD described by those papers in one region of spacetime, while a singularity of a dramatically different nature forms (such as an implosion, where the density itself is infinite) far away.

    Indeed, the recent results of \cite{merle2022implosion,merle2022implosion2} construct imploding solutions to the compressible Euler equations in $\mathbb{R}^{1+3}$ launched from  $C^\infty$ data. In particular, for those results, the location of the implosion would also be part of the boundary of the MGHD, but the acoustical metric would not extend continuously to it. 
    
    On the other hand, \cite{AbbresciaSpeck} proves that the singular boundary $\singularBoundary$ for non-degenerate shock forming data is an embedded $C^{1,1/2}$ hypersurface in $\R^{1+3}$, that this regularity is sharp, and that the acoustical metric extends as a $C^{0,1/2}$ Lorentzian metric to it. Is there initial data for which one could construct a singular boundary of lower regularity, or even worse, is only  rectifiable? 
    
    \item (\textbf{Inviscid limits}) There is a long history of works that investigate the vanishing-viscosity limit (VVL) of solutions to quasilinear hyperbolic conservation laws, especially for weak solutions in $1D$ that are allowed to have shocks. The basic hope is that the VVL of solutions to a family of a viscous models
    should yield the entropy solution to the inviscid equation in the limit, thereby helping to justify the
    physical relevance of the entropy solutions; see e.g. the classic papers
    \cite{GoodmanXin,Wang,BianchiniBressan} and book \cite{cDafermos}, which address the convergence rate in the norm $\| \cdot \|_{L^1}$ for general conservation laws $\partial_t \vec{u} + \partial_x (\vec{F}(\vec{u})) = \nu \vec{u}_{xx}$. This left open the following three problems: \textbf{I)} can you quantify the convergence in the VVL \emph{at the point of shock formation}? This is difficult because the viscosity regularizes the singularity, and once has to recover it in the limit; \textbf{II)} can you quantify the convergence in the VVL as a discontinuous shock emerges from the first point of gradient catastrophe as in the shock development problem?; \textbf{III)} can you quantify the convergence in the VVL in \textbf{I)}--\textbf{II)} in \emph{pointwise} norms? Regarding point \textbf{II)}, we clarify that \cite{GoodmanXin,Wang} worked in the setting where the discontinuous shock fronts were already present in the data assumptions and so they did not have to quantify how the shock initially developed. The seminal work of \cite{BianchiniBressan} dealt with arbitrary data which is small in BV.
    
    In a series of recent works, Anderson-Chaturvedi-Graham \cite{anderson2025shock,anderson2025shock2}  made spectacular progress in extending our understanding of how viscous solutions developed from shock-forming plane-symmetric data behave \emph{pointwise} in the VVL. In particular, for \emph{physical}\footnote{Many physically diffusive systems often do not feature viscosity in all equations as in $\partial_t \vec{u} + \partial_x (\vec{F}(\vec{u})) = \nu \vec{u}_{xx}$. For example, for the compressible Navier-Stokes equations, physical viscosity vanishes in the mass equation and dissipates the entropy density} viscous perturbations of hyperbolic conservation laws and small perturbations of plane-symmetric non-degenerate smooth shock-forming data, they proved that the viscous solution (which is regular) converges in $L^\infty$ to the inviscid shock-forming solution at a rate of $\nu^{1/4}$ \emph{up to and including the time of shock formation}. This is essentially a complete resolution of \textbf{I)} for this class of data. 

    This leaves the following interesting open problems: can you extend \cite{anderson2025shock,anderson2025shock2} to the shock development problem? Is there a way to describe the VVL up to the boundary of the MGHD for the inviscid solution? Note that, being diffusive PDEs, the viscous solutions do not have a natural Lorentzian structure associated to them as the inviscid equations, being hyperbolic, do! Is there a way to generalize \textbf{I)}--\textbf{II)} to higher dimensions?

\end{itemize}

\bibliography{QLWWeakSolutionShockInteruptingMGHD}
\bibliographystyle{alpha}


\printindex

\end{document}